\documentclass[smallcondensed]{svjour3}
\usepackage{latexsym}
\usepackage{amsmath}
\usepackage{amssymb}
\usepackage{enumerate}
\usepackage{setspace}
\usepackage{color}
\usepackage{graphicx}
\DeclareGraphicsExtensions{.pdf,.png,.jpg}
\usepackage{subfigure}
\usepackage[ruled]{algorithm}
\usepackage{algorithmic}
\usepackage{url,hyperref}
\usepackage{multicol}

\makeatletter
\renewcommand*\env@matrix[1][*\c@MaxMatrixCols c]{%
  \hskip -\arraycolsep
  \let\@ifnextchar\new@ifnextchar
  \array{#1}}
\makeatother

\newcommand{\R}{\mathbb{R}}

\DeclareMathOperator{\st}{s.t.}

\DeclareMathOperator{\Diag}{Diag}

\DeclareMathOperator{\myspan}{span}
\DeclareMathOperator{\rec}{rec.cone}
\DeclareMathOperator{\cone}{cone}
\DeclareMathOperator{\cvh}{conv.hull}
\DeclareMathOperator{\cnh}{conic.hull}
\DeclareMathOperator{\ccnh}{cl.conic.hull}
\DeclareMathOperator{\ccvh}{cl.conv.hull}

\DeclareMathOperator{\Null}{null}
\DeclareMathOperator{\Range}{Range}
\DeclareMathOperator{\apex}{apex}
\DeclareMathOperator{\tr}{trace}

\DeclareMathOperator{\bd}{bd}
\DeclareMathOperator{\myint}{int}

\newcommand{\eps}{\varepsilon}
\renewcommand{\S}{\mathcal{S}}

\newcommand{\E}{\mathcal{E}}
\newcommand{\F}{\mathcal{F}}
\newcommand{\G}{\mathcal{G}}
\newcommand{\Hc}{\mathcal{H}}

\newcommand{\K}{\mathcal{K}}

\newcommand{\Q}{\mathcal{Q}}

\newcommand{\T}{\mathcal{T}}

\newcommand{\be}{\begin{eqnarray}}
\newcommand{\ee}[1]{\label{#1}\end{eqnarray}}
\newcommand{\ese}{\end{eqnarray*}}
\newcommand{\bse}{\begin{eqnarray*}}

\newtheorem{condition}{Condition}
\newtheorem{observation}{Observation}

\journalname{Mathematical Programming}

\title{How to Convexify the Intersection \\ of a Second Order Cone \\
and a Nonconvex Quadratic}

\author{%
Samuel Burer
\and 
Fatma K{\i}l{\i}n\c{c}-Karzan
}

\institute{ S. Burer \at 
	Department of Management Sciences, University of Iowa, Iowa City, IA, 52242-1994, USA. \\
	\email{samuel-burer@uiowa.edu}
\and 
	F. K{\i}l{\i}n\c{c} Karzan \at
         Tepper School of Business, Carnegie Mellon University, Pittsburgh, PA, 15213, USA \\
          \email{fkilinc@andrew.cmu.edu}   
}

\date{Submitted on June 10, 2014; Revised on June 6, 2015 and May 19, 2016}

\begin{document}

\maketitle

\begin{abstract}

\noindent A recent series of papers has examined the extension of \sloppy
disjunctive-programming techniques to mixed-integer second-order-cone
programming. For example, it has been shown---by several authors using
different techniques---that the convex hull of the intersection of
an ellipsoid, $\E$, and a split disjunction, $(l - x_j)(x_j - u) \le
0$ with $l < u$, equals the intersection of $\E$ with an additional
second-order-cone representable (SOCr) set. In this paper, we study more
general intersections of the form $\K \cap \Q$ and $\K \cap \Q \cap H$,
where $\K$ is a SOCr cone, $\Q$ is a nonconvex cone defined by a single
homogeneous quadratic, and $H$ is an affine hyperplane. Under several
easy-to-verify conditions, we derive simple, computable convex
relaxations $\K \cap \S$ and $\K \cap \S \cap H$, where $\S$ is a SOCr
cone. Under further conditions, we prove that these two sets capture
precisely the corresponding conic/convex hulls. Our approach unifies
and extends previous results, and we illustrate its applicability and
generality with many examples.

\vspace{3mm}

\noindent {\bf Keywords:}
convex hull,
disjunctive programming,
mixed-integer linear programming,
mixed-integer nonlinear programming,
mixed-integer quadratic programming,
nonconvex quadratic programming,
second-order-cone programming,
trust-region subproblem.

\vspace{3mm}
\noindent {\bf Mathematics Subject Classification:} 90C25, 90C10, 90C11, 90C20, 90C26.

\end{abstract}


\section{Introduction}

In this paper, we study nonconvex intersections of the form $\K \cap
\Q$ and $\K \cap \Q \cap H$, where the cone $\K$ is second-order-cone
representable (SOCr), $\Q$ is a nonconvex cone defined by a single
homogeneous quadratic, and $H$ is an affine hyperplane. Our goal is
to develop tight convex relaxations of these sets and to characterize
the conic/convex hulls whenever possible. We are motivated by recent
research on Mixed Integer Conic Programs (MICPs), though our results
here enjoy wider applicability to nonconvex quadratic programs.

Prior to the study of MICPs in recent years, cutting plane theory has
been fundamental in the development of efficient and powerful solvers
for Mixed Integer Linear Programs (MILPs). In this theory, one considers
a convex relaxation of the problem, e.g., its continuous relaxation, and
then enforces integrality restrictions to eliminate regions containing
no integer feasible points---so-called \emph{lattice-free sets}. The
complement of a valid \emph{two-term linear disjunction}, say $x_j \le l \vee
x_j \ge u$, is a simple form of a lattice-free set. The additional
inequalities required to describe the convex hull of such a disjunction
are known as \emph{disjunctive cuts}. Such a disjunctive point of view
was introduced by Balas \cite{B1971} in the context of MILPs, and it has
since been studied extensively in mixed integer linear and nonlinear
optimization
\cite{%
Balas_79,Balas_Ceria_Cornuejols_93,Burer_Saxena_12,Cadoux_10,CS99,Cornuejols_Lemarechal_06,Kilinc_Linderoth_Luedtke_10,Saxena_Bonami_Lee_08,Sherali_Shetti_80%
}, complementarity \cite{Hu.et.al.2008,JSRF2006,NguyenTR11,Tawarmalani_Richard_Chung_10} and other nonconvex optimization
problems \cite{B2012,Burer_Saxena_12}. 
In the case of MILPs, several well-known classes of cuts such as
\emph{Chv{\'a}tal-Gomory}, \emph{lift-and-project}, \emph{mixed-integer
rounding (MIR)}, \emph{split}, and \emph{intersection cuts} are known to
be special types of disjunctive cuts. 
Stubbs and Mehrotra \cite{SM1999} and Ceria and Soares \cite{CS99} extended cutting plane theory from MILP to convex mixed integer problems. These works were followed by several papers 
\cite{%
B2011,D2009,Drewes_Pokutta_10,Kilinc_Linderoth_Luedtke_10,VAN08%
}
that investigated linear-outer-approximation based approaches, as
well as others that extended specific classes of inequalities, such as
Chv{\'a}tal-Gomory cuts \cite{CI2005} for MICPs and MIR cuts \cite{AV2010}
for SOC-based MICPs.

Recently there has been growing interest in developing closed-form
expressions for convex inequalities that fully describe the convex hull
of a disjunctive set involving an SOC. In this vein, G\"{u}nl\"{u}k
and Linderoth \cite{GL10perspective} studied a simple set involving
an SOC in $\R^3$ and a single binary variable and showed that the
resulting convex hull is characterized by adding a single SOCr
constraint. For general SOCs in $\R^n$, this line of work was furthered
by Dadush et al.\@ \cite{DDV2011}, who derived cuts for ellipsoids
based on parallel two-term disjunctions, that is, \emph{split
disjunctions}. Modaresi et al.\@ \cite{MKV} extended this by studying
\emph{intersection cuts} for SOC and all of its cross-sections (i.e.,
all conic sections), based on split disjunctions as well as a number
of other lattice-free sets such as ellipsoids and paraboloids. A
theoretical and computational comparison of intersection cuts from
\cite{MKV} with extended formulations and conic MIR inequalities from
\cite{AV2010} is given in \cite{MKV2}. Taking a different approach,
Andersen and Jensen \cite{AJ2013} derived an SOC constraint describing
the convex hull of a split disjunction applied to an SOC. Belotti et
al.\@ \cite{Belotti_Goez_Polik_Terlaky_12} studied families of quadratic
surfaces having fixed intersections with two given hyperplanes,
and in \cite{BGPRT}, they identified a procedure for constructing
two-term disjunctive cuts when the sets defined by the disjunctions
are bounded and disjoint. K{\i}l{\i}n\c{c}-Karzan \cite{KK} introduced
and examined minimal valid linear inequalities for general conic sets
with a disjunctive structure, and under a mild technical assumption,
established that they are sufficient to describe the resulting closed
convex hulls. For general two-term disjunctions on regular (closed,
convex, pointed with nonempty interior) cones, K{\i}l{\i}n\c{c}-Karzan
and Y{\i}ld{\i}z \cite{KY14} studied the structure of tight minimal
valid linear inequalities. In the particular case of SOCs, based on
conic duality, a class of convex valid inequalities that is sufficient
to describe the convex hull were derived in \cite{KY14} along with
the conditions for SOCr representability of these inequalities as
well as for the sufficiency of a single inequality from this class.
This work was recently extended in Y{\i}ld{\i}z and Cornu\'ejols
\cite{YC14} to all cross-sections of SOC that can be covered by the same
assumptions of \cite{KY14}. Bienstock and Michalka \cite{BM} studied
the characterization and separation of valid linear inequalities that
convexify the epigraph of a convex, differentiable function whose domain
is restricted to the complement of a convex set defined by linear or
convex quadratic inequalities. Although all of these authors take
different approaches, their results are comparable, for example, in the
case of analyzing split disjunctions of the SOC or its cross-sections.
We remark also that these methods convexify in the space of the
original variables, i.e., they do not involve lifting. For additional
convexification approaches for nonconvex quadratic programming, which
convexify in the lifted space of products $x_i x_j$ of variables, we
refer the reader to
\cite{Anstreicher2010computable,BaoST11,Burer2013second,Burer_Saxena_12,TawarmalaniRX13},
for example.

In this paper, our main contributions can be summarized as follows (see
Section \ref{sec:result} and Theorem \ref{the:main} in particular).
First, we derive a simple, computable convex relaxation $\K \cap \S$
of $\K \cap \Q$, where $\S$ is an additional SOCr cone. This
also provides the convex relaxation $\K \cap \S \cap H \supseteq \K
\cap \Q \cap H$. The derivation relies on several easy-to-verify
conditions (see Section \ref{sec:compdetails}). Second, we identify stronger conditions guaranteeing
moreover that
$\K \cap \S = \ccnh(\K \cap \Q)$ and $\K \cap \S \cap H =
\ccvh(\K \cap \Q \cap H)$, where {\em cl\/} indicates the closure, {\em
conic.hull\/} indicates the conic hull, and {\em conv.hull\/} indicates
the convex hull. 
Our approach unifies and significantly extends previous results. In particular, in contrast to the existing literature on cuts based on lattice-free sets, here we allow a general $\Q$ without making an assumption that $\R^n\setminus \Q$ is convex. 
We illustrate the applicability and generality of our approach with many examples and explicitly contrast our work with the existing literature.

Our approach can be seen as a variation of the following basic, yet
general, idea of conic aggregation to generate valid inequalities.
Suppose that $f_0 = f_0(x)$ is convex, while $f_1 = f_1(x)$ is
nonconvex, and suppose we are interested in the closed convex hull of
the set $Q:=\{x:\, f_0 \leq 0,\, f_1 \leq 0\}$. For any $0 \le t \le
1$, the inequality $f_t := (1-t) f_0 + t f_1 \leq 0$ is valid for $Q$,
but $f_t$ is generally nonconvex. Hence, it is natural to seek values
of $t$ such that the function $f_t$ is convex for all $x$. One might
even conjecture that some particular convex $f_s$ with $0 \le s \le 1$
guarantees $\ccvh(Q) = \{ x : f_0 \le 0, f_s \le 0 \}$. However, it is
known that this approach cannot generally achieve the convex hull even
when $f_0,f_1$ are quadratic functions; see \cite{MKV}. Such aggregation
techniques to obtain convex under-estimators have also been explored in
the global-optimization literature, albeit without explicit results on
the resulting convex hull descriptions (see \cite{Adjiman_etal98} for
example).

In this paper, we follow a similar approach in spirit, but instead of
determining $0 \le t \le 1$ guaranteeing the convexity of $f_t$ for all
$x$, we only require ``almost" convexity, that is, the function $f_t$ is
required to be convex on $\{x:f_0 \leq 0\}$. This weakened requirement
is crucial. In particular, it allows us to obtain convex hulls for many
cases where $\{ x : f_0 \le 0 \}$ is SOCr and $f_1$ is a nonconvex
quadratic, and we recover all of the known results regarding two-term
disjunctions cited above (see Section \ref{sec:twoterm}). We note
that using quite different techniques and under completely different
assumptions, a similar idea of aggregation for quadratic functions has
been explored in \cite{BGPRT,MKV} as well. Specifically, our weakened
requirement is in contrast to the developments in \cite{MKV}, which
explicitly requires the function $f_t$ to be convex everywhere. Also,
our general $\Q$ allows us to study general nonconvex quadratics $f_1$
as opposed to the specific ones arising from two-term disjunctions
studied in \cite{BGPRT}. As a practical and technical matter, instead
of working directly with convex functions in this paper, we work in
the equivalent realm of convex sets, in particular SOCr cones. Section
\ref{sec:socr} discusses in detail the features of SOCr cones required
for our analysis.

Compared to the previous literature on MICPs, our work here is
broader in that we study a general nonconvex cone $\Q$ defined by
a single homogeneous quadratic function. As a result, we assume
neither the underlying matrix defining the homogeneous quadratic
$\Q$ to be of rank at most 2 nor $\R^n\setminus\Q$ to be convex. 
 This is in contrast to a key underlying assumption used in the literature. 
Specifically, the majority of the earlier literature on MICPs
focus on specific lattice-free sets, e.g., all of the works
\cite{AJ2013,AV2010,BGPRT,DDV2011,KY14,YC14} focus on either split or
two-term disjunctions on SOCs or its cross-sections. In the case of
two-term disjunctions, the matrix defining the homogeneous quadratic
for $\Q$ is of rank at most 2; and moreover, the complement of any two-term disjunction is
a convex set. 
Even though, nonconvex quadratics $\Q$ with rank higher than 2 are considered in \cite{MKV}, unlike our general, $\Q$ this is done under the assumption that the complement of the nonconvex quadratic defines a convex set.  
Our general $\Q$ allows for a \emph{unified framework} and
works under weaker assumptions. 
In Sections \ref{sec:sub:ExEllipsoid} and \ref{sec:twoterm} and the Online Supplement, 
we illustrate and highlight these features of
our approach and contrast it with the existing literature through a
series of examples. Bienstock and Michalka \cite{BM} also consider
more general $\Q$ under the assumption that $\R^n\setminus \Q$ is
convex, but their approach is quite different than ours. Whereas
\cite{BM} relies on polynomial time procedures for separating and
tilting valid linear inequalities, we directly give the convex
hull description. In contrast, our study of the general, nonconvex
quadratic cone $\Q$ allows its complement $\R^n\setminus \Q$ to be
nonconvex as well.

We remark that our convexification tools for general nonconvex
quadratics have potential applications beyond MICPs, for example
in the nonconvex quadratic programming domain. We also can, for
example, characterize: the convex hull of the deletion of an
arbitrary ball from another ball; and the convex hull of the
deletion of an arbitrary ellipsoid from another ellipsoid sharing
the same center. In addition, we can use our results to solve the
classical trust region subproblem \cite{Conn.et.al.2000} using SOC
optimization, complementing previous approaches relying on nonlinear
\cite{Gould.et.al.1999,More.Sorensen.1983} or semidefinite programming
\cite{Rendl.Wolkowicz.1997}. Section \ref{sec:general} discusses these
examples.

Another useful feature of our approach is that we clearly distinguish
the conditions guaranteeing validity of our relaxation from those
ensuring sufficiency. In \cite{AJ2013,BGPRT,DDV2011,MKV}, validity
and sufficiency are intertwined making it difficult to construct
convex relaxations when their conditions are only partly satisfied.
Furthermore, our derivation of the convex relaxation is efficiently
computable and relies on conditions that are easily verifiable. Finally,
our conditions regarding the cross-sections (that is, intersection with
the affine hyperplane $H$) are applicable for general cones other than
SOCs.

We would like to stress that the inequality describing the SOCr set
$\S$ is efficiently computable. In other words, given the sets $\K \cap
\Q$ and $\K \cap \Q \cap \Hc$, one can verify in polynomial time the
required conditions and then calculate in polynomial time the inequality
for $\S$ to form the relaxations $\K \cap \S$ and $\K \cap \S \cap
H$. The core operations include calculating eigenvalues/eigenvectors
for several symmetric and non-symmetric matrices and solving a
two-constraint semidefinite program. The computation can also be
streamlined in cases when any special structure of $\K$ and $\Q$ is
known ahead of time.

The paper is structured as follows. Section \ref{sec:socr} discusses
the details of SOCr cones, and Section \ref{sec:result} states our
conditions and main theorem. In Section \ref{sec:compdetails}, we
provide a detailed discussion and pseudocode for verifying our
conditions and computing the resulting SOC based relaxation $\S$. 
Section \ref{sec:sub:ExEllipsoid} then provides a low-dimensional
example with figures and comparisons with existing literature. We provide more examples with corresponding figures and comparisons in the Online Supplement accompanying this article.    
In Section \ref{sec:proof}, we
prove the main theorem, and then in Sections \ref{sec:twoterm} and
\ref{sec:general}, we discuss and prove many interesting general
examples covered by our theory. Section \ref{sec:conclusion} concludes
the paper with a few final remarks. Our notation is mostly standard. We
will define any particular notation upon its first use.

\section{Second-Order-Cone Representable Sets} \label{sec:socr}

Our analysis in this paper is based on the concept of SOCr
(second-order-cone representable) cones. In this section, we define and
introduce the basic properties of such sets.

A cone $\F^+ \subseteq \R^n$ is said to be {\em second-order-cone
representable\/} (or {\em SOCr}) if there exists a matrix $0 \ne B \in
\R^{n \times (n-1)}$ and a vector $b \in \R^n$ such that the nonzero
columns of $B$ are linearly independent, $b \not\in \Range(B)$, and
\begin{equation} \label{equ:def:F+}
  \F^+ = \{ x : \|B^T x \| \le b^T x \},
\end{equation}
where $\|\cdot\|$ denotes the usual Euclidean norm. 
The negative of $\F^+$ is also SOCr:
\begin{equation} \label{equ:def:F-}
  \F^- := -\F^+ = \{ x : \|B^T x \| \le -b^T x \}.
\end{equation}
Defining $A := BB^T - bb^T$, the union $\F^+ \cup \F^-$ corresponds to
the homogeneous quadratic inequality $x^T A x \le 0$:
\begin{equation} \label{equ:def:F}
  \F := \F^+ \cup \F^- = \{ x : \|B^T x \|^2 \le (b^T x)^2 \} = \{ x : x^T A x \le 0 \}.
\end{equation}
We also define
\begin{align*}
  \myint(\F^+) &:= \{ x : \|B^T x \| < b^T x \} \\ 
  \bd(\F^+) &:= \{ x : \|B^T x \| = b^T x \} \\ 
  \apex(\F^+) &:= \{ x : B^T x = 0, b^T x = 0 \}.  
\end{align*}

We next study properties of $\F,\F^+,\F^-$ such as their
representations and uniqueness thereof. On a related note, Mahajan
and Munson \cite{MM10} have also studied sets associated with
nonconvex quadratics with a single negative eigenvalue but from a more
computational point of view. The following proposition establishes some
important features of SOCr cones:

\begin{proposition} \label{pro:basicprops}
Let $\F^+$ be SOCr as in (\ref{equ:def:F+}), and define $A := BB^T -
bb^T$. Then $\apex(\F^+) = \Null(A)$, $A$ has at least one
positive eigenvalue, and $A$ has exactly one negative eigenvalue. As a
consequence, $\myint(\F^+) \ne \emptyset$.
\end{proposition}

\begin{proof}
For any $x$, we have the equation
\begin{equation} \label{equ:local2}
Ax = (BB^T - bb^T)x = B(B^T x) - b(b^T x).
\end{equation}
So $x \in \apex(\F^+)$ implies $x \in \Null(A)$. The converse also holds
by (\ref{equ:local2}) because, by definition, the nonzero columns of
$B$ are independent and $b \not\in \Range(B)$. Hence, $\apex(\F^+) =
\Null(A)$. 

The equation $A = BB^T - bb^T$, with $0 \ne BB^T \succeq 0$ and $bb^T
\succeq 0$ rank-1 and $b \not\in \Range(B)$, implies that $A$ has at
least one positive eigenvalue and at most one negative eigenvalue.
Because $b \not\in \Range(B)$, we can write $b = x + y$ such that $x \in
\Range(B)$, $0 \ne y \in \Null(B^T)$, and $x^T y = 0$. Then
\[
    y^T A y = y^T (BB^T - bb^T) y = 0 - (b^Ty)^2 = -\|y\|^2 < 0,
\]
showing that $A$ has exactly one negative eigenvalue, and so
$\myint(\F^+)$ contains either $y$ or $-y$. 
\qed \end{proof}

\noindent We define analogous sets $\myint(\F^-)$, $\bd(\F^-)$, and
$\apex(\F^-)$ for $\F^-$. In addition:
\begin{align*}
  \myint(\F) &:= \{ x : x^T A x < 0 \} = \myint(\F^+) \cup \myint(\F^-)\\
  \bd(\F) &:= \{ x : x^T A x = 0 \} = \bd(\F^+) \cup \bd(\F^-).
\end{align*}
Similarly, we have $\apex(\F^-) = \Null(A)=\apex(\F^+)$, and if $A$ has exactly one
negative eigenvalue, then $\myint(\F^-) \ne \emptyset$ and $\myint(\F)
\ne \emptyset$.

When considered as a pair of sets $\{\F^+,\F^-\}$, it is possible that
another choice $(\bar B, \bar b)$ in place of $(B,b)$ leads to the
same pair and hence to the same $\F$. For example, $(\bar B, \bar b) =
(-B,-b)$ simply switches the roles of $\F^+$ and $\F^-$, but $\F$ does
not change. However, we prove next that $\F$ is essentially invariant
up to positive scaling. As a corollary, any alternative $(\bar B, \bar
b)$ yields $A = \rho(\bar B \bar B^T - \bar b \bar b^T)$ for some $\rho
> 0$, i.e., $A$ is essentially invariant with respect to its $(B,b)$
representation.

\begin{proposition}\label{prop:representationUniqueness}
Let $A,\bar{A}$ be two $n\times n$ symmetric matrices such that
$\{x\in\R^n:~ x^TAx\leq 0\}=\{x\in\R^n:~ x^T\bar{A}x\leq 0\}$. Suppose
that $A$ satisfies $\lambda_{\min}(A)<0<\lambda_{\max}(A)$. Then there
exists $\rho> 0$ such that $\bar{A}=\rho A$.
\end{proposition}

\begin{proof}
Since $\lambda_{\min}(A) < 0$, there exists $\bar x \in \R^n$ such
that $\bar x^T A \bar x < 0$. Because $x^TAx \leq 0 \Leftrightarrow
x^T\bar{A}x\leq 0$, there exists no $x$ such that $x^T A x \le 0$
and $x^T (-\bar A) x < 0$. Then, by the S-lemma (see Theorem 2.2 in
\cite{Polik.Terlaky.2007}, for example), there exists $\lambda_1 \ge 0$
such that $-\bar A + \lambda_1 A \succeq 0$. Switching the roles of $A$
and $\bar A$, a similar argument implies the existence of $\lambda_2
\ge 0$ such that $-A + \lambda_2 \bar A \succeq 0$. Note $\lambda_2
> 0$; otherwise, $A$ would be negative semidefinite, contradicting
$\lambda_{\max}(A) > 0$. Likewise, $\lambda_1 > 0$.
Hence,
\[
A \succeq \frac{1}{\lambda_1} \bar A \succeq \frac{1}{\lambda_1 \lambda_2} A
\ \Longleftrightarrow \ (1 - \lambda_1 \lambda_2) A \succeq 0.
\]
Since $\lambda_{\min}(A) < 0 < \lambda_{\max}(A)$, we conclude
$\lambda_1 \lambda_2 = 1$, which in turn implies $A =
\frac{1}{\lambda_1} \bar A$, as claimed. \qed
\end{proof}

\begin{corollary}
Let $\{\F^+,\F^-\}$ be SOCr sets as in (\ref{equ:def:F+}) and
(\ref{equ:def:F-}), and define $A := BB^T - bb^T$. Let $(\bar B,\bar
b)$ be another choice in place of $(B,b)$ leading to the same pair $\{
\F^+,\F^- \}$. Then $A = \rho(\bar B \bar B^T - \bar b \bar b^T)$ for
some $\rho > 0$.
\end{corollary}

We can reverse the discussion thus far to start from a symmetric
matrix $A$ with at least one positive eigenvalue and a single negative
eigenvalue and define associated SOCr cones $\F^+$ and $\F^-$.
Indeed, given such an $A$, let $Q \Diag(\lambda) Q^T$ be a spectral
decomposition of $A$ such that $\lambda_1 < 0$. Let $q_j$ be the $j$-th
column of $Q$, and define
\begin{equation} \label{equ:def:Bb}
  B :=
  \begin{pmatrix}
    \lambda_{2}^{1/2} q_{2} & \cdots & \lambda_{n}^{1/2} q_{n} 
  \end{pmatrix} \in \R^{n \times (n-1)},
  \ \ \ \
  b := (-\lambda_{1})^{1/2} q_{1} \in \R^n.
\end{equation}
Note that the nonzero columns of $B$ are linearly independent and $b
\not\in \Range(B)$. Then $A = BB^T - bb^T$, and $\F = \F^+ \cup \F^-$
can be defined as in (\ref{equ:def:F+})--(\ref{equ:def:F}). An important
observation is that, as a collection of sets, $\{ \F^+,\F^- \}$ is
independent of the choice of spectral decomposition.

\begin{proposition} \label{pro:invariant}
Let $A$ be a given symmetric matrix with at least one positive
eigenvalue and a single negative eigenvalue, and let $A = Q
\Diag(\lambda) Q^T$ be a spectral decomposition such that $\lambda_1
< 0$. Define the SOCr sets $\{ \F^+,\F^- \}$ according to
(\ref{equ:def:F+}) and (\ref{equ:def:F-}), where $(B,b)$ is given by
(\ref{equ:def:Bb}). Similarly, let $\{\bar\F^+,\bar\F^-\}$ be defined
by an alternative spectral decomposition $A = \bar Q \Diag(\bar\lambda)
\bar Q^T$. Then $\{\bar \F^+,\bar\F^-\} = \{\F^+,\F^-\}$.
\end{proposition}

\begin{proof}
Let $(\bar B, \bar b)$ be given by the alternative spectral
decomposition. Because $A$ has a single negative eigenvalue, $\bar b =
b$ or $\bar b = -b$. In addition, we claim $\| \bar B^T x \| = \| B^T
x\|$ for all $x$. This holds because $\bar B \bar B^T = BB^T$ is the
positive semidefinite part of $A$. This proves the result.
\qed \end{proof}

\noindent To resolve the ambiguity inherent in Proposition
\ref{pro:invariant}, one could choose a specific $\bar x \in
\myint(\F)$, which exists by Proposition \ref{pro:basicprops}, and
enforce the convention that, for any spectral decomposition, $\F^+$ is
chosen to contain $\bar x$. This simply amounts to flipping the sign of
$b$ so that $b^T \bar x > 0$.

\section{The Result and Its Computability} \label{sec:result}

In Section \ref{sec:subresult}, we state our main theorem (Theorem
\ref{the:main}) and the conditions upon which it is based. The proof
of Theorem \ref{the:main} is delayed until Section \ref{sec:proof}. In
Section \ref{sec:compdetails}, we discuss computational details
related to our conditions and Theorem \ref{the:main}.

\subsection{The result} \label{sec:subresult}

To begin, let $A_0$ be a symmetric matrix satisfying the following:

\begin{condition} \label{ass:one_neg_eval}
$A_0$ has at least one positive eigenvalue and exactly one negative
eigenvalue.
\end{condition}

\noindent As described in Section \ref{sec:socr}, we may define SOCr
cones $\F_0 = \F_0^+ \cup \F_0^-$ based on $A_0$. We also introduce a
symmetric matrix $A_1$ and define the cone $\F_1 := \{ x : x^T A_1 x
\le 0 \}$ in analogy with $\F_0$. However, we do {\em not\/} assume
that $A_1$ has exactly one negative eigenvalue, so $\F_1$ does not
necessarily decompose into two SOCr cones.

We investigate the set $\F_0^+ \cap \F_1$, which has been expressed
as $\K \cap \Q$ in the Introduction.
In particular, we would like to develop strong convex relaxations of
$\F_0^+ \cap \F_1$ and, whenever possible, characterize its closed conic
hull. We focus on the full-dimensional case, and so we assume:

\begin{condition} \label{ass:interior}
There exists $\bar x \in \myint(\F_0^+ \cap \F_1)$.
\end{condition}

\noindent Note that $\myint(\F_0^+ \cap \F_1) = \myint(\F_0^+)
\cap \myint(\F_1)$, and so Condition \ref{ass:interior} is equivalent to 
\begin{equation} \label{equ:neginnprod}
  \bar x^T A_0 \bar x < 0
  \ \ \ \ \ \text{and} \ \ \ \ \
  \bar x^T A_1 \bar x < 0.
\end{equation}
In particular, this implies $A_1$ has at least one negative eigenvalue.

The first part of Theorem \ref{the:main} below establishes that
$\ccnh(\F_0^+ \cap \F_1)$ is contained within the convex intersection of
$\F_0^+$ with a second set of the same type, i.e., one that is SOCr. In
addition to Conditions \ref{ass:one_neg_eval} and \ref{ass:interior}, we
require the following condition, which handles the singularity of $A_0$
carefully via several cases:

\begin{condition} \label{ass:A0_apex}
Either (i) $A_0$ is nonsingular, (ii) $A_0$ is singular and $A_1$ is
positive definite on $\Null(A_0)$, or (iii) $A_0$ is singular and $A_1$
is negative definite on $\Null(A_0)$.
\end{condition}

Conditions \ref{ass:one_neg_eval}--\ref{ass:A0_apex} will ensure
(see Proposition \ref{pro:spos} in Section \ref{sec:proof:interval})
the existence of a maximal $s \in [0,1]$ such that $$A_t:=(1-t)A_0+t
A_1$$ has a single negative eigenvalue for all $t \in [0,s]$, $A_t$
is invertible for all $t \in (0,s)$, and $A_s$ is singular---that
is, $\Null(A_s)$ is non-trivial. (Actually, $A_s$ may be nonsingular
when $s$ equals 1, but this is a small detail.) Indeed, we define $s$
formally as follows. Let $\T := \{ t \in \R : A_t \text{ is singular}
\}$. Then
\begin{equation} \label{equ:def:s}
  s :=
  \left\{
    \begin{array}{ll}
      \min( \T \cap (0,1] ) & \text{under Condition \ref{ass:A0_apex}(i) or \ref{ass:A0_apex}(ii)} \\
      0 & \text{under Condition \ref{ass:A0_apex}(iii).}
    \end{array}
  \right.
\end{equation}
Sections \ref{sec:compdetails} and \ref{sec:proof} will clarify the role
of Condition \ref{ass:A0_apex} in this definition.

With $s$ given by (\ref{equ:def:s}), we can then define, for all $A_t$
with $t \in [0,s]$, SOCr sets $\F_t = \F_t^+ \cup \F_t^-$ as described
in Section \ref{sec:socr}. Furthermore, for $\bar x$ of Condition
\ref{ass:interior}, noting that $\bar x^T A_t \bar x = (1-t) \, \bar x
A_0 \bar x + t \, \bar x^T A_1 \bar x < 0$ by (\ref{equ:neginnprod}),
we can choose without loss of generality that $\bar x \in \F_t^+$ for
all such $t$. Then Theorem \ref{the:main} asserts that $\ccnh(\F_0^+
\cap \F_1)$ is contained in $\F_0^+ \cap \F_s^+$. We remark that while
$\F_0^+ \cap \F_1\subseteq \F_0^+ \cap \F_s$ (no ``$+$'' superscript on
$\F_s$) follows trivially from the definition of $\F_s$, strengthening
the inclusion to $\F_0^+ \cap \F_1\subseteq \F_0^+ \cap \F_s^+$ (with
the ``$+$'' superscript) is nontrivial.

The second part of Theorem \ref{the:main} provides an additional
condition under which $\F_0^+ \cap \F_s^+$ actually equals the closed
conic hull. The required condition is:

\begin{condition} \label{ass:As_null}
When $s < 1$, $\apex(\F_s^+) \cap \myint(\F_1) \ne \emptyset$.
\end{condition}

\noindent While Condition \ref{ass:As_null} may
appear quite strong, we will actually show (see Lemma
\ref{lem:As_null} in Section \ref{sec:proof}) that Conditions
\ref{ass:one_neg_eval}--\ref{ass:A0_apex} and the definition of
$s$ already ensure $\apex(\F_s^+) \subseteq \F_1$. So Condition
\ref{ass:As_null} is a type of regularity condition guaranteeing that
the set $\apex(\F_s^+) = \Null(A_s)$ is not restricted to the boundary
of $\F_1$.

We also include in Theorem \ref{the:main} a specialization for the case
when $\F_0^+ \cap \F_1$ is intersected with an affine hyperplane $H^1$,
which has been expressed as $\K \cap \Q \cap H$ in the Introduction. For
this, let $h \in \R^n$ be given, and define the hyperplanes
\begin{align}
  H^1 &:= \{ x : h^T x = 1 \}, \label{equ:def:H1} \\
  H^0 &:= \{ x : h^T x = 0 \}. \label{equ:def:H0}
\end{align}
We introduce an additional condition related to $H^0$:

\begin{condition} \label{ass:hyperplane}
When $s < 1$, $\apex(\F_s^+)\cap \myint(\F_1) \cap H^0 \ne \emptyset$
or $\F_0^+ \cap \F_s^+ \cap H^0 \subseteq \F_1$.
\end{condition}

We now state the main theorem of the paper. See Section \ref{sec:proof}
for its proof.

\begin{theorem} \label{the:main}
Suppose Conditions \ref{ass:one_neg_eval}--\ref{ass:A0_apex}
are satisfied, and let $s$ be defined by (\ref{equ:def:s}). Then
$\ccnh(\F_0^+ \cap \F_1) \subseteq \F_0^+ \cap \F_s^+$, and
equality holds under Condition \ref{ass:As_null}. Moreover, 
Conditions
\ref{ass:one_neg_eval}--\ref{ass:hyperplane} imply
$
  \F_0^+ \cap \F_s^+ \cap H^1 = \ccvh(\F_0^+ \cap \F_1 \cap H^1).
$
\end{theorem}

\subsection{Computational details} \label{sec:compdetails}

In practice, Theorem \ref{the:main} can be used to generate a valid
convex relaxation $\F_0^+ \cap \F_s^+$ of the nonconvex cone $\F_0^+
\cap \F_1$. For the purposes of computation, we assume that $\F_0^+ \cap
\F_1$ is described as
\[
    \F_0^+ \cap \F_1 = \{ x \in \R^n : \| B_0^T x \| \le b_0^T x, \ x^T A_1 x \le 0 \},
\]
where $B_0$ is nonzero, $0 \ne b_0 \not\in \Range(B_0)$, and $A_0 = B_0
B_0^T - b_0 b_0^T$ in accordance with (\ref{equ:def:Bb}). In particular,
$\F_0^+$ is given in its direct SOC form. Our goal is to calculate
$\F_s^+$ in terms of its SOC form $\| B_s^T x \| \le b_s^T x$, to which
we will refer as the SOC cut.

Before one can apply Theorem \ref{the:main} to generate the cut,
Conditions 1--3 must be verified. By construction, Condition
\ref{ass:one_neg_eval} is satisfied, and verifying Condition
\ref{ass:A0_apex}(i) is easy. Conditions \ref{ass:A0_apex}(ii) and
\ref{ass:A0_apex}(iii) are also easy to verify by computing the
eigenvalues of $Z_0^T A_1 Z_0$, where $Z_0$ is a matrix whose columns
span $\Null(A_0)$. Due to (\ref{equ:neginnprod}) and the fact that
$\F_0$ and $\F_1$ are cones, verifying Condition \ref{ass:interior}
is equivalent to checking the feasibility of the following quadratic
equations in the original variables $x\in\R^n$ and the auxiliary ``squared
slack'' variables $s, t \in \R$:
\[
x^TA_0 x + s^2= -1, \ x^TA_1 x + t^2= -1 .
\]
Let us define the underlying symmetric $(n+2)\times(n+2)$ matrices for
these quadratics as $\hat{A}_0$ and $\hat{A}_1$. Since there are only
two quadratic equations with symmetric matrices, by \cite[Corollary
13.2]{Barvinok}, checking Condition \ref{ass:interior} is equivalent
to checking the feasibility of the following linear semidefinite
system, which can be done easily in practice: \begin{equation}
\label{equ:Barvinok} Y\succeq 0, ~\tr(\hat{A}_0 Y) =-1, ~ \tr(\hat{A}_1
Y) =-1. \end{equation} See also \cite{Pataki98} for a similar result.

This equivalence of Condition \ref{ass:interior} and the feasibility of
system (\ref{equ:Barvinok}) relies on the fact that every extreme point
of (\ref{equ:Barvinok}) is a rank-1 matrix, and such extreme points
can be calculated in polynomial time \cite{Pataki98}. Extreme points
can also be generated reliably (albeit heuristically) in practice to
calculate an interior point $\bar x \in \myint(\F_0^+ \cap \F_1)$. One
can simply minimize over (\ref{equ:Barvinok}) the objective $\tr((I +
R)Y)$, where $I$ is the identity matrix and $R$ is a random matrix,
small enough so that $I + R$ remains positive definite. The objective
$\tr((I + R)Y)$ is bounded over (\ref{equ:Barvinok}), and hence an
optimal solution occurs at an extreme point. The random nature of the
objective also makes it highly likely that the optimal solution is
unique, in which case the optimal $Y^*$ must be rank-1. Then $\bar x$
can easily be extracted from the rank-1 factorization of $Y^*$. Note
that in certain specific cases $\bar{x}$ might be known ahead of time or
could be computed right away by some other means.

Once Conditions \ref{ass:one_neg_eval}--\ref{ass:A0_apex} have
been verified, we are then ready to calculate $s$ according to its
definition (\ref{equ:def:s}). If Condition \ref{ass:A0_apex}(iii)
holds, we simply set $s = 0$. For Conditions \ref{ass:A0_apex}(i) and
\ref{ass:A0_apex}(ii), we need to calculate $\T$, the set of scalars
$t$ such that $A_t := (1 - t) A_0 + t A_1$ is singular. Let us first
consider Condition \ref{ass:A0_apex}(i), which is the simpler case.
The following calculation with $t \ne 0$ shows that the elements
of $\T$ are in bijective correspondence with the real eigenvalues of
$A_0^{-1} A_1$:
\begin{align*}
  A_t \text{ is singular}
  \ \ \ &\Longleftrightarrow \ \ \ 
  \exists \ x \ne 0 \ \st \ A_t x = 0 \\
  \ \ \ &\Longleftrightarrow \ \ \ 
  \exists \ x \ne 0 \ \st \ A_0^{-1} A_1 x = - \left(\tfrac{1-t}{t}\right) x \\
  \ \ \ &\Longleftrightarrow \ \ \ 
  -\left(\tfrac{1-t}{t}\right) \text{ is an eigenvalue of } A_0^{-1} A_1.
\end{align*}
So to calculate $\T$, we calculate the real eigenvalues $\E$ of
$A_0^{-1} A_1$, and then calculate $\T = \{ (1-e)^{-1} : e \in \E \}$,
where by convention $0^{-1} = \infty$. In particular, $|\T|$ is finite.

When Condition \ref{ass:A0_apex}(ii) holds, we calculate $\T$ in a
slightly different manner. We will show in Section \ref{sec:proof} (see
Lemma \ref{lem:Ae} in particular) that, even though $A_0$ is singular,
$A_\epsilon$ is nonsingular for all $\epsilon > 0$ sufficiently small.
Such an $A_\epsilon$ could be calculated by systematically testing
values of $\epsilon$ near 0, for example. Then we can apply the
procedure of the previous paragraph to calculate the set $\overline
\T$ of all $\bar t$ such that $(1 - \bar t)A_\epsilon + \bar t A_1$ is
singular. Then one can check that $\T$ is calculated by the following
affine transformation: $\T = \{ (1 - \epsilon) \bar t + \epsilon : \bar t
\in \overline \T \}$.

Once $\T$ is computed, we can easily calculate $s = \min(\T \cap (0,1])$
according to (\ref{equ:def:s}), and then we construct $A_s := (1 - s)A_0
+ s A_1$ and calculate $(B_s, b_s)$ according to (\ref{equ:def:Bb}).
Then our cut is $\| B_s^T x \| \le b_s^T x$ with only one final
provision. We must check the sign of $b_s^T \bar x$, where $\bar x \in
\myint(\F_0^+ \cap \F_1)$ has been calculated previously. If $b_s^T \bar
x \ge 0$, then the cut is as stated; if $b_s^T \bar x < 0$, then the cut
is as stated but $b_s$ is first replaced by $-b_s$.

We summarize the preceding discussion by the pseudocode in Algorithm
\ref{alg:cut}. While this algorithm is quite general, it is also
important to point out that it can be streamlined if one already knows
the structure of $\| B_0^T x \| \le b_0^T x$ and $x^T A_1 x \le 0$. For
example, one may already know that $A_0$ is invertible, in which case it
would be unnecessary to calculate the spectral decomposition of $A_0$
in Algorithm \ref{alg:cut}. 
In addition, for many of the specific cases
that we consider in Sections \ref{sec:twoterm} and \ref{sec:general},
we can explicitly point out the corresponding value of $s$ without
even relying on the computation of the set $\T$. Because of space
considerations, we do not include these closed-form expressions for $s$ and the corresponding computations.

\begin{algorithm}
\caption{Calculate Cut (see also Section \ref{sec:compdetails})} \label{alg:cut}
\begin{algorithmic}[1]
    \REQUIRE Inequalities $\|B_0^T x\| \le b_0^T x$ and $x^T A_1 x \le 0$.
    \ENSURE Valid cut $\| B_s^T x \| \le b_s^T x$.
    \STATE Calculate $A_0 = B_0 B_0^T - b_0 b_0^T$ and a spectral
    decomposition $Q_0 \Diag(\lambda_0) Q_0^T$. Let $Z_0$ be the submatrix
    of $Q_0$ of zero eigenvectors (possibly empty).
    \STATE Minimize $\tr((I+R)Y$ over (\ref{equ:Barvinok}). If
    infeasible, then STOP. Otherwise, extract $\bar x \in \myint(\F_0^+
    \cap \F^1)$ from $Y^*$.
    \IF{$Z_0$ is empty}
        \STATE Calculate the set ${\cal E}$ of real eigenvalues
        of $A_0^{-1} A_1$.
        \STATE Set ${\cal T} = \{ (1-e)^{-1} : e \in {\cal E} \}$.
        \STATE Set $s = \min(\T \cap (0,1])$.
    \ELSIF{$Z_0^T A_1 Z_0 \succ 0$}
        \STATE Determine $\epsilon > 0$ small such that $A_\epsilon =
        (1 - \epsilon)A_0 + \epsilon A_1$ is invertible.
        \STATE Calculate the set $\overline{{\cal E}}$ of real eigenvalues
        of $A_\epsilon^{-1} A_1$.
        \STATE Set $\overline{\T} = \{ (1-\bar e)^{-1} :
        \bar e \in \overline{\cal E} \}$.
        \STATE Set $\T = \{ (1 - \epsilon)\bar t + \epsilon : \bar t \in
        \overline{\cal T} \}$.
        \STATE Set $s = \min(\T \cap (0,1])$.
    \ELSIF{$Z_0^T A_1 Z_0 \prec 0$}
        \STATE Set $s = 0$.
    \ELSE
        \STATE STOP.
    \ENDIF
    \STATE Calculate $A_s = B_s B_s^T - b_s b_s^T$ and a spectral
    decomposition $Q_s \Diag(\lambda_s) Q_s^T$. Let $(B_s,b_s)$ be given
    by (\ref{equ:def:Bb}).
    \STATE If $b_s^T \bar x < 0$, replace $b_s$ by $-b_s$.
\end{algorithmic}
\end{algorithm}

Finally, we mention briefly the computability of Conditions
\ref{ass:As_null} and \ref{ass:hyperplane}, which are not necessary for
the validity of the cut but can establish its sufficiency. Given $s <
1$, Condition \ref{ass:As_null} can be checked by computing $Z_s^T A_1
Z_s$, where $Z_s$ has columns spanning $\Null(A_s)$. We know $Z_s^T
A_1 Z_s \preceq 0$ because $\apex(\F_s^+) \subseteq \F_1$ (see Lemma
\ref{lem:As_null} in Section \ref{sec:proof}), and then Condition
\ref{ass:As_null} holds as long as $Z_s^T A_1 Z_s \ne 0$. On the other
hand, it seems challenging to verify Condition \ref{ass:hyperplane} in
general. However, in Sections \ref{sec:twoterm} and \ref{sec:general},
we will show that it can be verified in many examples of interest.


\subsection{An ellipsoid and a nonconvex quadratic}\label{sec:sub:ExEllipsoid}

In $\R^3$, consider the intersection of the unit ball defined by
$y_1^2 + y_2^2 + y_3^2 \le 1$ and the nonconvex set defined by
the quadratic $-y_1^2 - y_2^2 + \tfrac12 y_3^2 \le y_1 + \tfrac12 y_2$.
By homogenizing via $x = {y \choose
x_4}$ with $x_4 = 1$, we can represent the intersection as $\F_0^+ \cap
\F_1 \cap H^1$ with
\[
  A_0 :=
  \begin{pmatrix}
    1 & 0 & 0 & 0         \\
    0 & 1 & 0 & 0         \\
    0 & 0 & 1 & 0 \\
    0 & 0 & 0 & -1
  \end{pmatrix}, \ \ \
  A_1 :=
  \begin{pmatrix}
    -1        &  0         & 0        & -\tfrac12 \\
    0         & -1         & 0        & -\tfrac14 \\
    0         &  0         & \tfrac12 & 0 \\
    -\tfrac12 &  -\tfrac14 & 0        & 0 
  \end{pmatrix}, \ \ \
  H^1 := \{ x : x_4 = 1 \}.
\]
Conditions \ref{ass:one_neg_eval} and \ref{ass:A0_apex}(i) are
straightforward to verify, and Condition \ref{ass:interior} is
satisfied with $\bar x = (\tfrac12;0;0;1)$, for
example. We can also calculate $s = \tfrac12$ from (\ref{equ:def:s}). Then
\[
  A_s =
  \tfrac18
  \begin{pmatrix}
    0  &  0  & 0 & -2 \\
    0  &  0  & 0 & -1 \\
    0  &  0  & 6 & 0 \\
    -2 &  -1 & 0 & -4
  \end{pmatrix}, \ \ \ \ \
  \F_s = \left\{ x : 3 x_3^2 \le 2 x_1 x_4 + x_2 x_4 + 2 x_4^2 \right\}.
\]
The negative eigenvalue of $A_s$ is $\lambda_{s1} := -\tfrac58$
with corresponding eigenvector $q_{s1} := (2;1;0;5)$, and so,
in accordance with the Section \ref{sec:socr}, we have that $\F_s^+$
equals all $x \in \F_s$ satisfying $b_s^T x \ge 0$, where
\[
  b_s := (-\lambda_{s1})^{1/2} q_{s1}
  = \sqrt{5/8}
  \begin{pmatrix}
    2 \\ 1 \\ 0 \\ 5
  \end{pmatrix}.
\]
In other words,
\[
  \F_s^+ := \left\{ x : \begin{array}{ll}
      3 x_3^2 \le 2 x_1 x_4 + x_2 x_4 + 2 x_4^2 \\
      2 x_1 + x_2 + 5 x_4 \ge 0
    \end{array}
  \right\}.
\]
Note that $\bar x \in \F_s^+$. In addition, $\apex(\F_s^+) = \Null(A_s)
= \myspan\{d\}$, where $d = (1;-2;0;0)$. Clearly, $d \in H^0$ and
$d^T A_1 d < 0$, which verifies Conditions \ref{ass:As_null} and
\ref{ass:hyperplane} simultaneously. Setting $x_4 = 1$ and returning to
the original variables $y$, we see
\[
  \left\{ y : 
  \begin{array}{l}
    y_1^2 + y_2^2 + y_3^2 \le 1 \\
    3 y_3^2 \le 2 y_1 + y_2 + 2 
  \end{array}
  \right\}
  =
  \ccvh\left\{ y :
  \begin{array}{l}
    y_1^2 + y_2^2 + y_3^2 \le 1 \\
    -y_1^2 - y_2^2 + \tfrac12 y_3^2 \le y_1 + \tfrac12 y_2
  \end{array}
  \right\},
\]
where the now redundant constraint $2y_1 + y_2 \ge -5$ has been dropped.
Figure \ref{fig:sphere_quad} depicts the original set, $\F_s^+\cap H^1$, and the closed convex
hull.

\begin{figure}[htp]
\centering
  \subfigure[$\F_0^+ \cap \F_1 \cap H^1$]{%
    \label{sphere_quad_F0F1}%
    \includegraphics[width=0.30\textwidth]{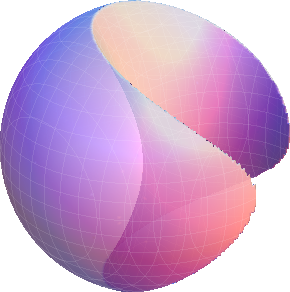}%
  } 
  \subfigure[$\F_s^+ \cap H^1$]{%
    \label{sphere_quad_Fs}%
    \includegraphics[width=0.36\textwidth]{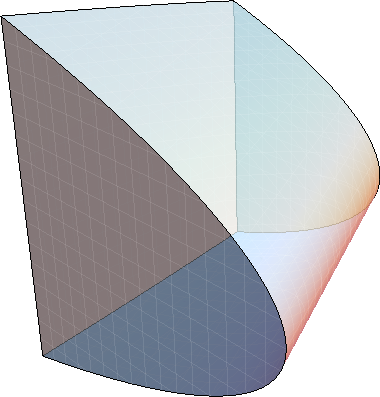}%
  }
  \subfigure[$\F_0^+ \cap \F_s^+ \cap H^1$]{%
    \label{sphere_quad_F0Fs}%
    \includegraphics[width=0.30\textwidth]{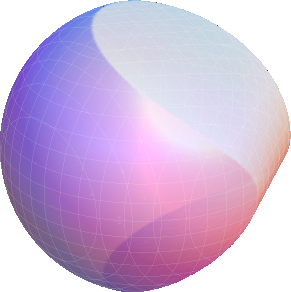}%
  }
\caption{An ellipsoid and a nonconvex quadratic}
\label{fig:sphere_quad}
\end{figure}

Of the earlier, related approaches, this example
can be handled by \cite{MKV} only. In particular,
\cite{AJ2013,BGPRT,DDV2011,K-KY14,KY14,YC14} cannot handle this example
because they deal with only split or two-term disjunctions but cannot
cover general nonconvex quadratics. The approach of \cite{BM} is based
on eliminating a convex region from a convex epigraphical set, but this
example removes a nonconvex region (specifically, $\R^n\setminus\F_1$).
So \cite{BM} cannot handle this example either.

In actuality, the results of \cite{MKV} do not handle this example
explicitly since the authors only state results for: the removal of a
paraboloid or an ellipsoid from a paraboloid; or the removal of an ellipsoid
(or an ellipsoidal cylinder) from another ellipsoid with a common
center. However, in this particular example, the function obtained from
the aggregation technique described in \cite{MKV} is convex on all of
$\R^3$. Therefore, their global convexity requirement on the aggregated
function is satisfied for this example.


\section{The Proof} \label{sec:proof}

In this section, we build the proof of Theorem \ref{the:main}, and we
provide important insights along the way. The key results are
Propositions \ref{pro:containment}--\ref{pro:H1}, which state
\begin{align*}
&\F_0^+ \cap \F_1 \subseteq
\F_0^+ \cap \F_s^+ \subseteq
\cnh(\F_0^+ \cap \F_1) \\
&\F_0^+ \cap \F_1 \cap H^1   \subseteq
\F_0^+ \cap \F_s^+ \cap H^1 \subseteq
\cvh(\F_0^+ \cap \F_1 \cap H^1),
\end{align*}
where $s$ is given by (\ref{equ:def:s}).
In each line here, the first containment depends only on Conditions
\ref{ass:one_neg_eval}--\ref{ass:A0_apex}, which proves the first
part of Theorem \ref{the:main}. On the other hand, the second
containments require Condition \ref{ass:As_null} and Conditions
\ref{ass:As_null}--\ref{ass:hyperplane}, respectively. Then the second
part of Theorem \ref{the:main} follows by simply taking the closed conic
hull and the closed convex hull, respectively, and noting that $\F_0^+
\cap \F_s^+$ and $\F_0^+ \cap \F_s^+ \cap H^1$ are already closed and
convex.

\subsection{The interval $[0,s]$} \label{sec:proof:interval}

Our next result, Lemma \ref{lem:Ae}, is quite technical but critically
important. For example, it establishes that the line of matrices $\{
A_t \}$ contains at least one invertible matrix not equal to $A_1$.
As discussed in Section \ref{sec:result}, this proves that the set
$\T$ used in the definition (\ref{equ:def:s}) of $s$ is finite and
easily computable. The lemma also provides additional insight into
the definition of $s$. Specifically, the lemma clarifies the role of
Condition \ref{ass:A0_apex} in (\ref{equ:def:s}).

\begin{lemma} \label{lem:Ae}
For $\epsilon > 0$ small, consider $A_\epsilon$ and
$A_{-\epsilon}$. Relative to Condition \ref{ass:A0_apex}:
\begin{itemize}

  \item if (i) holds, then $A_\epsilon$ and
  $A_{-\epsilon}$ are each invertible with one negative
  eigenvalue;
 
  \item if (ii) holds, then only $A_\epsilon$ is invertible with 
  one negative eigenvalue;

  \item if (iii) holds, then only $A_{-\epsilon}$ is invertible with
  one negative eigenvalue.

\end{itemize}
\end{lemma}

\noindent Since the proof of Lemma \ref{lem:Ae} is involved, we delay it
until the end of this subsection.

If Condition \ref{ass:A0_apex}(i) or \ref{ass:A0_apex}(ii)
holds, then Lemma \ref{lem:Ae} shows that the interval $(0,\epsilon)$
contains invertible $A_t$, each with exactly one negative eigenvalue,
and (\ref{equ:def:s}) takes $s$ to be the largest $\epsilon$ with
this property. By continuity, $A_s$ is singular (when $s < 1$) but still retains
exactly one negative eigenvalue, a necessary condition for defining
$\F_s^+$ in Theorem \ref{the:main}. On the other hand, if Condition
\ref{ass:A0_apex}(iii) holds, then $A_0$ is singular and no $\epsilon >
0$ has the property just mentioned. Yet, $s=0$ is still the natural
``right-hand limit'' of invertible $A_{-\epsilon}$, each with exactly
one negative eigenvalue. This will be all that is required for Theorem
\ref{the:main}.

With Lemma \ref{lem:Ae} in hand, we can prove the following key
result, which sets up the remainder of this section. The proof of
Lemma \ref{lem:Ae} follows afterwards.

\begin{proposition} \label{pro:spos}
Suppose Conditions \ref{ass:one_neg_eval}--\ref{ass:A0_apex} hold.
For all $t \in [0,s]$, $A_t$ has exactly one negative eigenvalue. In
addition, $A_t$ is nonsingular for all $t \in (0,s)$, and if $s < 1$,
then $A_s$ is singular.
\end{proposition}

\begin{proof}
Condition \ref{ass:interior} implies (\ref{equ:neginnprod}), and so
$\bar x^T A_t \bar x = (1-t)\, \bar x^T A_0 \bar x + t \, \bar x^T A_1
\bar x < 0$ for every $t$. So each $A_t$ has at least one negative
eigenvalue. Also, the definition of $s$ ensures that all $A_t$ for $t
\in (0,s)$ are nonsingular and that $A_s$ is singular when $s < 1$.

Suppose that some $A_t$ with $t \in [0,s]$ has two negative
eigenvalues. Then by Condition \ref{ass:one_neg_eval} and the
facts that the entries of $A_t$ are affine functions of $t$ and the
eigenvalues depend continuously on the matrix entries \cite[Section
2.4.9]{Horn_Johnson_2013}, there exists some $0 \le r < t \le s$ with
at least one zero eigenvalue, i.e., with $A_r$ singular. From the
definition of $s$, we deduce that $r = 0$ and $A_\epsilon$
has two negative eigenvalues for $\epsilon > 0$ small. Then Condition
\ref{ass:A0_apex}(ii) holds since $s > 0$. However, we then encounter a
contradiction with Lemma \ref{lem:Ae}, which states that $A_\epsilon$
has exactly one negative eigenvalue. \qed
\end{proof}

\begin{proof}[of Lemma \ref{lem:Ae}]
The lemma holds under Condition \ref{ass:A0_apex}(i) since
$A_0$ is invertible with exactly one negative eigenvalue and  the
eigenvalues are continuous in $\epsilon$.

Suppose Condition \ref{ass:A0_apex}(ii) holds. Let $V$ be the subspace
spanned by the zero and positive eigenvectors of $A_0$, and consider
\[
    \theta := \inf \{ x^T A_0 x : x^T (A_0 - A_1) x = 1, x \in V \}.
\]
Clearly $\theta \ge 0$, and we claim $\theta > 0$. If $\theta = 0$,
then there exists $\{ x^k \} \subseteq V$ with $(x^k)^T A_0 x^k \to 0$
and $(x^k)^T(A_0 - A_1) x^k = 1$ for all $k$.
If $\{x^k\}$ is bounded, then passing to a subsequence if necessary,
we have $x^k \to \hat x$ such that $\hat x^T A_0 \hat x = 0$ and $\hat
x^T(A_0 - A_1) \hat x = 1$, which implies $\hat x^T A_1 \hat x = -1$,
a contradiction of Condition \ref{ass:A0_apex}(ii). On the other hand,
if $\{x^k\}$ is unbounded, then the sequence $d^k := x^k/\|x^k\|$ is
bounded, and
passing to a subsequence if necessary, we see that $d^k \to \hat d$ with
$\|\hat d \| = 1$, $\hat d^T A_0 \hat d = 0$ and $\hat d^T (A_0 - A_1)
\hat d = 0$. This implies $\hat d^T A_1 \hat d = 0$, violating Condition
\ref{ass:A0_apex}(ii). So $\theta > 0$.

Now choose any $0 < \epsilon \le \theta/2$, and take any nonzero $x \in V$.
Note that
\begin{equation}\label{eq:AepsFormula}
    x^T A_\epsilon x = (1 - \epsilon)x^T A_0 x + \epsilon x^T A_1 x
    = x^T A_0 x - \epsilon x^T (A_0 - A_1) x.
\end{equation}
We wish to show $x^T A_\epsilon x > 0$, and so we consider three
subcases. First, if $x^T (A_0 - A_1) x = 0$, then it must hold that
$x^T A_0 x > 0$. If not, then $x^T A_1 x = 0$ also, violating Condition
\ref{ass:A0_apex}(ii). So $x^T A_\epsilon x = x^T A_0 x > 0$. Second, if
$x^T (A_0 - A_1) x < 0$, then because $x\in V$ we have $x^T A_\epsilon x
> 0$. Third, if $x^T (A_0 - A_1) x > 0$, then we may assume without loss
of generality by scaling that $x^T (A_0 - A_1) x = 1$ in which case $x^T
A_\epsilon x \ge \theta - \epsilon > 0$.

So we have shown that $A_\epsilon$ is positive definite on a subspace of
dimension $n-1$, which implies that $A_{\epsilon}$ has at least $n-1$
positive eigenvalues. In addition, we know that $A_{\epsilon}$ has at
least one negative eigenvalue because $\bar x^T A_\epsilon \bar x < 0$
according to Condition \ref{ass:interior} and (\ref{equ:neginnprod}).
Hence, $A_\epsilon$ is invertible with exactly one negative eigenvalue,
as claimed.

By repeating a very similar argument for vectors $x\in W$, the subspace
spanned by the negative and zero eigenvectors of $A_0$ (note that $W$ is
at least two-dimensional because Condition \ref{ass:A0_apex}(ii) holds),
and once again using the relation \eqref{eq:AepsFormula}, we can show
that $A_{-\epsilon}$ has at least two negative eigenvalues, as claimed.

Finally, suppose Condition \ref{ass:A0_apex}(iii) holds and define
\begin{align*}
\bar A_\epsilon &:= 
  \left( \tfrac{1}{1 + 2\epsilon} \right) A_{-\epsilon} =
  \left( \tfrac{1}{1 + 2\epsilon} \right) \left( (1+\epsilon) A_0 - \epsilon A_1 \right) =
  \left( \tfrac{1 + \epsilon}{1 + 2\epsilon} \right) A_0 +
  \left( \tfrac{\epsilon}{1 + 2\epsilon} \right) (-A_1) \\
\bar A_{-\epsilon} &:= 
  \left( \tfrac{1}{1 - 2\epsilon} \right) A_\epsilon =
  \left( \tfrac{1}{1 - 2\epsilon} \right) \left( (1-\epsilon) A_0 + \epsilon A_1 \right) =
  \left( \tfrac{1 - \epsilon}{1 - 2\epsilon} \right) A_0 +
  \left( \tfrac{-\epsilon}{1 - 2\epsilon} \right) (-A_1). 
\end{align*}
Then $\bar A_{\epsilon}$ and $\bar A_{-\epsilon}$ are on the line generated
by $A_0$ and $-A_1$ such that $-A_1$ is positive definite on the null
space of $A_0$. Applying the previous case for Condition \ref{ass:A0_apex}(ii), we
see that only $\bar A_{\epsilon}$ is invertible with a single negative
eigenvalue. This proves the result.
\qed \end{proof}

\subsection{The containment $\F_0^+ \cap \F_1 \subseteq \F_0^+ \cap \F_s^+$}

For each $t \in [0,s]$, Proposition \ref{pro:spos} allows us to
define analogs $\F_t = \F_t^+ \cup \F_t^-$ as described in Section
\ref{sec:socr} based on any spectral decomposition $A_t = Q_t
\Diag(\lambda_t) Q_t^T$.

It is an important technical point, however, that in this paper we
require $\lambda_t$ and $Q_t$ to be defined continuously in $t$. While
it is well known that the vector of eigenvalues $\lambda_t$ can be
defined continuously, it is also known that---if the eigenvalues are
ordered, say, such that $[\lambda_t]_1 \le \cdots \le [\lambda_t]_n$ for
all $t$---then the corresponding eigenvectors, i.e., the ordered columns
of $Q_t$, cannot be defined continuously in general. On the other hand,
if one drops the requirement that the eigenvalues in $\lambda_t$ stay
ordered, then the following result of Rellich \cite{Rellich.1969} (see
also \cite{Kato.1976}) guarantees that $\lambda_t$ and $Q_t$ can be
constructed continuously---in fact, analytically---in $t$:

\begin{theorem}[Rellich \cite{Rellich.1969}] \label{the:Rellich}
Because $A_t$ is analytic in the single parameter $t$, there exist
spectral decompositions $A_t = Q_t \Diag(\lambda_t) Q_t^T$ such that
$\lambda_t$ and $Q_t$ are analytic in $t$.
\end{theorem}

So we define $\F_t^+$ and $\F_t^-$ using continuous spectral
decompositions provided by Theorem \ref{the:Rellich}:
\begin{align*}
  \F_t^+ &:= \{ x : \|B_t^T x\| \le b_t^T x \} \\
  \F_t^- &:= \{ x : \|B_t^T x\| \le -b_t^T x \},
\end{align*}
where $B_t$ and $b_t$ such that $A_t = B_t B_t^T - b_t b_t^T$ are
derived from the spectral decomposition as described in Section
\ref{sec:socr}. Recall from Proposition \ref{pro:invariant} that, for
each $t$, a different spectral decomposition could flip the roles of
$\F_t^+$ and $\F_t^-$, but we now observe that Theorem \ref{the:Rellich}
and Condition \ref{ass:interior} together guarantee that each $\F_t^+$
contains $\bar x$ from Condition \ref{ass:interior}. In this sense, every
$\F_t^+$ has the same ``orientation.'' Our observation is enabled by a
lemma that will be independently helpful in subsequent analysis.

\begin{lemma} \label{lem:bt^Tx=0}
Suppose Conditions \ref{ass:one_neg_eval}--\ref{ass:A0_apex} hold.
Given $t \in [0,s]$, suppose some $x \in \F_t^+$ satisfies $b_t^T x =
0$. Then $t = 0$ or $t = s$.
\end{lemma}

\begin{proof}
Since $x^T A_t x \le 0$ with $b_t^T x = 0$, we have
$  0 = (b_t^T x)^2 \ge \| B_t^T x \|^2$ which implies
$  A_t x = (B_t B_t^T - b_t b_t^T) x = B_t (B_t^T x) - b_t (b_t^T x) = 0.$ 
So $A_t$ is singular. By Proposition \ref{pro:spos}, this implies $t=0$
or $t = s$.
\qed \end{proof}

\begin{observation} \label{obs:xbar}
Suppose Conditions \ref{ass:one_neg_eval}--\ref{ass:A0_apex} hold.
Let $\bar x \in \myint(\F_0^+\cap\F_1)$. Then for all $t \in [0,s]$, $\bar x \in \F_t^+$.
\end{observation}

\begin{proof}
Condition \ref{ass:interior} implies $b_0^T \bar x > 0$. Let
$t \in (0,s]$ be fixed. Since $\bar x^T A_t \bar x < 0$ by
(\ref{equ:neginnprod}), either $\bar x \in \F_t^+$ or $\bar x
\in \F_t^-$. Suppose for contradiction that $\bar x \in \F_t^-$,
i.e., $b_t^T \bar x < 0$. Then the continuity of $b_t$ by Theorem
\ref{the:Rellich} implies the existence of $r \in (0,t)$ such that
$b_r^T \bar x = 0$. Because $\bar x^T A_r \bar x < 0$ as well, $\bar x
\in \F_r^+$. By Lemma \ref{lem:bt^Tx=0}, this implies $r=0$ or $r=s$, a
contradiction.
\qed \end{proof}

\noindent In particular, Observation \ref{obs:xbar} implies that our
discussion in Section \ref{sec:result} on choosing $\bar x \in
\F_t^+$ to facilitate the statement of Theorem \ref{the:main} is indeed
consistent with the discussion here.

The primary result of this subsection,  
$\F_0^+ \cap \F_s^+$ is a valid convex relaxation of
$\F_0^+ \cap \F_1$, is given below.

\begin{proposition} \label{pro:containment}
Suppose Conditions \ref{ass:one_neg_eval}--\ref{ass:A0_apex} hold. Then 
$\F_0^+ \cap \F_1 \subseteq \F_0^+ \cap \F_s^+$.
\end{proposition}

\begin{proof}
If $s = 0$, the result is trivial. So assume $s > 0$. In particular,
Condition \ref{ass:A0_apex}(i) or \ref{ass:A0_apex}(ii) holds. Let $x
\in \F^+_0 \cap \F_1$, that is, $x^T A_0 x \le 0$, $b_0^T x \ge 0$, and
$x^T A_1 x \le 0$. We would like to show $x \in \F^+_0 \cap \F^+_s$.
So we need $x^T A_s x \le 0$ and $b_s^T x \ge 0$. The first inequality
holds because $x^T A_s x = (1-s) \, x^T A_0 x + s \, x^T A_1 x \le 0$.
Now suppose for contradiction that $b_s^T x < 0$. In particular, $x
\ne 0$. Then by the continuity of $b_t$ via Theorem \ref{the:Rellich},
there exists $0 \le r < s$ such that $b_r^T x = 0$. Since $x^T A_r x \le
0$ also, $x \in \F_r^+$, and Lemma \ref{lem:bt^Tx=0} implies $r = 0$.
So Condition \ref{ass:A0_apex}(ii) holds. However, $x \in \F_1$ also,
contradicting that $A_1$ is positive definite on $\Null(A_0)$.
\qed \end{proof}

\subsection{The containment $\F_0^+ \cap \F_s^+ \subseteq \cnh(\F_0^+ \cap \F_1)$}

Proposition \ref{pro:containment} in the preceding
subsection establishes that $\F_0^+ \cap \F_s^+$ is a valid
convex relaxation of $\F_0^+ \cap \F_1$ under Conditions
\ref{ass:one_neg_eval}--\ref{ass:A0_apex}. We now show that, in
essence, the reverse inclusion holds under Condition \ref{ass:As_null}
(see Proposition \ref{pro:clconichull}). Indeed, when $s = 1$, we
clearly have $\F_0^+ \cap \F_1^+ \subseteq \F_0^+ \cap \F_1 \subseteq
\cnh(\F_0^+ \cap \F_1)$. So the true case of interest is $s < 1$, for
which Condition \ref{ass:As_null} is the key ingredient. (However,
results are stated to cover the cases $s < 1$ and $s=1$ simultaneously.)

As mentioned in Section \ref{sec:result}, Condition \ref{ass:As_null} is
a type of regularity condition in light of Lemma \ref{lem:As_null} next.
The proof of Proposition \ref{pro:clconichull} also relies on Lemma
\ref{lem:As_null}.

\begin{lemma} \label{lem:As_null}
Suppose Conditions \ref{ass:one_neg_eval}--\ref{ass:A0_apex} hold. Then
$\apex(\F_s^+) \subseteq \F_1$.
\end{lemma}

\begin{proof}
By Proposition \ref{pro:basicprops}, the claimed result is equivalent to
$\Null(A_s) \subseteq \F_1$. Let $d \in \Null(A_s)$. If $s = 1$, then
$d^T A_1 d = 0$, i.e., $d \in \bd(\F_1) \subseteq \F_1$, as desired. If
$s = 0$, then Condition \ref{ass:A0_apex}(iii) holds, that is, $A_0$ is
singular and $A_1$ is negative definite on $\Null(A_0)$. Then $d \in
\Null(A_0)$ implies $d^T A_1 d \le 0$, as desired.

So assume $s \in (0,1)$. If $d \not\in \myint(\F_0)$, that is, $d^T A_0
d \ge 0$, then the equation $0 = (1-s) \, d^T A_0 d + s \, d^T A_1 d$
implies $d^T A_1 d \le 0$, as desired.

We have thus reduced to the case $s \in (0,1)$ and $d \in \myint(\F_0)$,
and we proceed to derive a contradiction. Without loss of generality,
assume that $d \in \myint(\F_0^+)$ and $-d \in \myint(\F_0^-)$. We know
$-d \in \Null(A_s) = \apex(\F_s^+) \subseteq \F_s^+$. In total, we have
$-d \in \F_s^+ \cap \myint(\F_0^-)$. We claim that, in fact, $\F_t^+
\cap \myint(\F_0^-) \ne \emptyset$ as $t \to s$. 

Note that $\F_t^+$ is a full-dimensional set because $\bar x^T A_t
\bar x < 0$ by (\ref{equ:neginnprod}). Also, $\F_t^+$ is defined
by the intersection of a homogeneous quadratic $x^TA_tx\leq 0$ and a
linear constraint $b_t^Tx\geq 0$ and $(A_t,b_t) \to (A_s,b_s)$ as $t
\to s$. Then the boundary of $\F_t^+$ converges to the boundary of
$\F_s^+$ as $t \to s$. Since $\F_t^+$ is a full-dimensional, convex set
(in fact SOC), $\F_t^+$ then converges as a set to $\F_s^+$ as $t \to
s$. So there exists a sequence $y_t \in \F_t^+$ converging to $-d$. In
particular, $\F_t^+ \cap \myint(\F_0^-) \ne \emptyset$ for $t \to s$.

We can now achieve the desired contradiction. For $t < s$,
let $x \in \F_t^+ \cap \myint(\F_0^-)$. Then $x^T A_0 x \le 0,\ 
b_0^T x < 0$ and $x^T A_t x \le 0,\  b_t^T x \ge 0$. It follows that $x^T
A_r x \le 0,\  b_r^T x = 0$ for some $0 < r \le t < s$. Hence, Lemma
\ref{lem:bt^Tx=0} implies $r = 0$ or $r = s$, a contradiction.
\qed \end{proof}

\begin{proposition} \label{pro:clconichull}
Suppose Conditions \ref{ass:one_neg_eval}--\ref{ass:As_null} hold. Then
$\F_0^+ \cap \F_s^+ \subseteq \cnh(\F_0^+ \cap \F_1)$.
\end{proposition}

\begin{proof}
First, suppose $s = 1$. Then the result follows because $\F_0^+ \cap
\F_1^+ \subseteq \F_0^+ \cap \F_1 \subseteq \cnh(\F_0^+ \cap \F_1)$. So
assume $s \in [0,1)$.

Let $x \in \F^+_0 \cap \F^+_s$, that is, $x^T A_0 x \le 0$, $b_0^T x \ge
0$ and $x^T A_s x \le 0$, $b_s^T x \ge 0$. If $x^T A_1 x \le 0$, we are
done. So assume $x^T A_1 x > 0$. 

By Condition \ref{ass:As_null}, there
exists $d \in \Null(A_s)$ such that $d^T A_1 d < 0$. In addition, $d$
is necessarily perpendicular to the negative eigenvector $b_s$. For all
$\epsilon \in \R$, consider the affine line of points given by $x_\eps
:= x + \epsilon \, d$. We have
\[
  \left.
  \begin{array}{r}
x_\eps^T A_s x_\eps =  (x + \epsilon \, d)^T A_s (x + \epsilon \, d) = x^T A_s x \le 0 \\
b_s^T x_\eps =  b_s^T (x + \epsilon \, d) = b_s^T x \ge 0
\end{array}
\right\}
\ \ \ \Longrightarrow \ \ \ x_\epsilon \in \F^+_s.
\]
Note that $x_\eps^T A_1 x_\eps = x^T A_1 x + 2 \, \epsilon \, d^TA_1 x +
\epsilon^2 \, d^T A_1 d$. Then $x_\eps^T A_1 x_\eps$ defines a quadratic
function of $\eps$ and its roots are given by
$ \eps_{\pm}=\frac{-d^TA_1 x\pm\sqrt{(d^TA_1 x)^2-(x^T A_1 x)(d^T A_1 d)}}{d^T A_1 d}.$  
Since $x^T A_1 x > 0$ and $d^T A_1 d < 0$, the discriminant is
greater than $|d^T A_1 x|$. Hence, one of the roots will be
positive and the other one will be negative. Then there exist $l:=\eps_-
< 0 < \eps_+=:u$ such that $x_{l}^T A_1 x_{l}^T = x_{u}^T A_1 x_{u} =
0$, i.e., $x_{l},x_{u}\in\F_1$. Then $s < 1$ and $x_{l}^T A_s x_{l} \le
0$ imply $x_{l}^T A_0 x_{l} \le 0$, and hence $x_{l}\in\F_0$. Similarly,
$x_{u}^T A_0 x_{u} \le 0$ leading to $x_{u}\in\F_0$. We will prove in
the next paragraph that both $x_l$ and $x_u$ are in $\F_0^+$, which will
establish the result because then $x_{l}, x_{u} \in \F^+_0 \cap \F^1$
and $x$ is a convex combination of $x_{l}$ and $x_{u}$.

Suppose that at least one of the two points $x_l$ or $x_u$ is not a
member of $\F_0^+$. Without loss of generality, say $x_l \not \in
\F_0^+$. Then $x_l \in \F^-_0$ with $-b_0^T x_l > 0$. Similar to
Proposition \ref{pro:containment}, we can prove $\F_0^- \cap \F_1
\subseteq \F_0^- \cap \F_s^-$, and so $x_l \in \F_0^- \cap \F_s^-$.
Then $x_l \in \F_s^+ \cap \F_s^-$, which implies $b_s^T x_l = 0$ and
$B_s^T x_l = 0$, which in turn implies $A_s x_l = 0$, i.e., $x_l \in
\Null(A_s)$. Then $x + l \, d = x_l  \in \Null(A_s)$ implies $x \in \Null(A_s)$ also.
Then $x \in \F_1$ by Lemma \ref{lem:As_null}, but this contradicts the
earlier assumption that $x^T A_1 x > 0$.
\qed \end{proof}

\subsection{Intersection with an affine hyperplane}

As discussed at the beginning of this section, Propositions
\ref{pro:containment}-\ref{pro:clconichull} allow us to prove the first
two statements of Theorem \ref{the:main}. In this subsection, we prove
the last statement of the theorem via Proposition \ref{pro:H1} below.
Recall that $H^1$ and $H^0$ are defined according to (\ref{equ:def:H1})
and $(\ref{equ:def:H0})$, where $h \in \R^n$. Also define
\[
  H^+ := \{ x : h^T x \ge 0 \}.
\]

Our first task is to prove the analog of Propositions
\ref{pro:containment}--\ref{pro:clconichull} under intersection with
$H^+$. Specifically, we wish to show that the inclusions

\begin{equation} \label{equ:containmentH+}
\F_0^+ \cap \F_1 \cap H^+   \subseteq
\F_0^+ \cap \F_s^+ \cap H^+ \subseteq
\cnh(\F_0^+ \cap \F_1 \cap H^+)
\end{equation}
hold under Conditions \ref{ass:one_neg_eval}--\ref{ass:hyperplane}. As
Condition \ref{ass:hyperplane} consists of two parts, we break the proof
into two corresponding parts (Lemma \ref{lem:halfspace1} and Corollary
\ref{cor:halfspace2}). Note that Condition \ref{ass:hyperplane} only
applies when $s < 1$, although results are stated covering both $s < 1$
and $s=1$ simultaneously.

\begin{lemma} \label{lem:halfspace1}
Suppose Conditions \ref{ass:one_neg_eval}--\ref{ass:As_null}
and the first part of Condition \ref{ass:hyperplane} hold. Then
(\ref{equ:containmentH+}) holds.
\end{lemma}

\begin{proof}
Proposition \ref{pro:containment} implies that $\F_0^+ \cap \F_1 \cap
H^+ \subseteq \F_0^+ \cap \F_s^+ \cap H^+$. Moreover, we can repeat the
proof of Proposition \ref{pro:clconichull}, intersecting with $H^+$
along the way. However, we require one key modification in the proof of
Proposition \ref{pro:clconichull}.

Let $x \in \F_0^+ \cap \F_s^+ \cap H^+$ with $x^T A_1 x > 0$. Then,
mimicking the proof of Proposition \ref{pro:clconichull} for $s \in
[0,1)$ and $d \in \apex(\F_s^+) \cap \myint(\F_1)$ from Condition \ref{ass:As_null}, $x \in \{ x_\epsilon := x + \epsilon \, d : \epsilon \in \R
\} \subseteq \F_s^+$.
Moreover, $x$ is a strict convex combination of points $x_l, x_u
\in \F_0^+ \cap \F_1$ where $x_l, x_u$ are as defined in the proof of Proposition \ref{pro:clconichull}. Hence, the entire closed interval from $x_l$
to $x_u$ is contained in $\F_0^+ \cap \F_s^+$.

Under the first part of Condition \ref{ass:hyperplane}, if there exists
$d \in\apex(\F_s^+)\cap \myint(\F_1) \cap H^0$, then $h^Td=0$ and this
particular $d$ can be used to show that $x_l, x_u$ identified in the
proof of Proposition \ref{pro:clconichull} also satisfy $h^Tx_l =h^T(x +
l \, d) = h^T x \ge 0$ (recall that $x\in H^+$) and $h^Tx_u = h^T(x + u
\, d) =h^T x \ge 0$, i.e., $x_l, x_u\in \F_0^+ \cap \F_1 \cap H^+$. Then
this implies $x \in \F_0^+ \cap \F_1 \cap H^+$, as desired.
\qed \end{proof}

\noindent Regarding the second part of Condition \ref{ass:hyperplane},
we prove Corollary \ref{cor:halfspace2} using the following more general
lemma involving cones that are not necessarily SOCr:

\begin{lemma}\label{lem:cross-section}
Let $\G_0$, $\G_1$, and $\G_s$ be cones such that $\G_0, \G_s$
are convex, $\G_0 \cap \G_1 \subseteq \G_0\cap \G_s \subseteq
\cnh(\G_0\cap\G_1)$ and $\G_0\cap\G_s\cap H^0 \subseteq \G_1$. Then
$$\G_0 \cap \G_1 \cap H^+ \subseteq \G_0\cap\G_s \cap H^+ \subseteq
\cnh(\G_0\cap\G_1 \cap H^+).$$
\end{lemma}

\begin{proof}
For notational convenience, define $\G_{01} := \G_0 \cap \G_1$ and
$\G_{0s} := \G_0 \cap \G_s$. We clearly have $\G_{01} \cap H^+ \subseteq
\G_{0s} \cap H^+ \subseteq \cnh(\G_{01}) \cap H^+$. We will show
$\G_{0s} \cap H^+ \subseteq \cnh(\G_{01} \cap H^+)$. Consider $x\in
\G_{0s} \cap H^+$. Either $h^T x = 0$ or $h^T x >0$.

If $h^T x = 0$, then $x \in \G_{0s} \cap H^0 \subseteq \G_1$ by
the premise of the lemma. Thus $x \in \G_{0s} \cap H^+ \cap \G_1 \subseteq
\cnh(\G_{01} \cap H^+)$, as desired.

When $h^T x > 0$, because $\G_{0s} \subseteq \cnh(\G_{01})$, we know
that $x$ can be expressed as a finite sum $x = \sum_k \lambda_k x^k$,
where each $x^k \in \G_{01} \subseteq \G_{0s}$ and $\lambda_i > 0$.
Define $I := \{ k : h^T x^k \ge 0 \}$ and $J := \{ k : h^T x^k < 0
\}$. If $J = \emptyset$, then we are done as we have shown $x \in
\cnh(\G_{01} \cap H^+)$. If not, then for all $j \in J$, let $y^j$ be
a strict conic combination of $x$ and $x^j$ such that $y^j \in H^0$.
In particular, there exists $\alpha_j \ge 0$ and $\beta_j > 0$ such
that $y^j = \alpha_j x + \beta_j x^j$. Note also that $y^j \in \G_{0s}$
because $\G_{0s}$ is convex and $x,x^j \in \G_{0s}$. Then $y^j \in
\G_{0s} \cap H^0 \subseteq \G_1$. As a result, for all $j\in J$, we have
$y^j \in\G_{01} \cap H^+$. Rewriting $x$ as
\[
x
= \sum_{i \in I} \lambda_i x^i + \sum_{j\in J}{\lambda_j \over
\beta_j} \left(y^j- \alpha_j x\right) \quad \Longleftrightarrow\quad
\left(1+\sum_{j\in J}{\lambda_j\alpha_j \over \beta_j}\right) x =
\sum_{i \in I} \lambda_i x^i + \sum_{j \in J}{\lambda_j \over \beta_j} y^j,
\]
we conclude that $x$ is a conic combination of points in $\G_{01} \cap
H^+$, as desired.
\qed \end{proof}

\begin{corollary} \label{cor:halfspace2}
Suppose Conditions \ref{ass:one_neg_eval}--\ref{ass:As_null} and
the second part of Condition \ref{ass:hyperplane} hold. Then
(\ref{equ:containmentH+}) holds.
\end{corollary}

\begin{proof}
Apply Lemma \ref{lem:cross-section} with $\G_0 := \F_0^+$,
$\G_1 := \F_1$, and $\G_s := \F_s^+$. Propositions
\ref{pro:containment}--\ref{pro:clconichull} and the second part of
Condition \ref{ass:hyperplane} ensure that the hypotheses of Lemma
\ref{lem:cross-section} are met. Then the result follows.
\qed \end{proof}

Even though our goal in this subsection is Proposition \ref{pro:H1},
which involves intersection with the hyperplane $H^1$, we remark
that Lemmas \ref{lem:halfspace1}--\ref{lem:cross-section} can help
us investigate intersections with homogeneous halfspaces $H^+$ for
SOCr cones (Lemma \ref{lem:halfspace1}) or more general cones (Lemma
\ref{lem:cross-section}). Further, by iteratively applying Lemmas
\ref{lem:halfspace1}--\ref{lem:cross-section}, we can consider
intersections with multiple halfspaces, say, $H_1^+, \ldots, H_m^+$.

Given Lemma \ref{lem:halfspace1} and Corollary
\ref{cor:halfspace2}, we are now ready to prove our main result for this
subsection, Proposition \ref{pro:H1}, which establishes the second part
of Theorem \ref{the:main}. It requires the following simple lemmas which
are applicable to general sets and cones:

\begin{lemma} \label{lem:reccone}
Let $S$ be any set, and let $\rec(S)$ be its recession cone. Then
$\cvh(S) + \cnh(\rec(S)) = \cvh(S)$.
\end{lemma}
\begin{proof}
The containment $\supseteq$ is clear. Now let $x+y$ be in the left-hand
side such that
\bse
  x &=& \sum_k \lambda_k x_k, \ \ \ x_k \in S, \ \ \ \lambda_k > 0, \ \ \ \sum_k \lambda_k = 1, \\
\mbox{ and }\quad
  y &=& \sum_j \rho_j y_j, \ \ \ y_j \in \rec(S), \ \ \ \rho_j > 0.
\ese
Without loss of generality, we may assume the number of $x_k$'s equals
the number of $y_j$'s by splitting some $\lambda_k x_k$ or some $\rho_j
y_j$ as necessary. Then
\[
  x + y = \sum_k (\lambda_k x_k + \rho_k y_k) = \sum_k \lambda_k ( x_k + \lambda_k^{-1}
  \rho_k y_k) \in \cvh(S). \ \ \ \ \qed
\]
\end{proof}

\begin{lemma} \label{lem:coneCross-section}
Let $\G_{01}$ and $\G_{0s}$ be cones (not necessarily convex) such that
$\G_{01} \cap H^+ \subseteq \G_{0s} \cap H^+ \subseteq \cnh(\G_{01}
\cap H^+)$. Then $\G_{01} \cap H^1 \subseteq \G_{0s} \cap H^1 \subseteq
\cvh(\G_{01}\cap H^1)$.
\end{lemma}

\begin{proof}
We have $\G_{01}\cap H^1 \subseteq \G_{0s}\cap H^1\subseteq
\cnh(\G_{01}\cap H^+) \cap H^1$. We claim further that
\begin{equation} \label{equ:local10}
\cnh(\G_{01}\cap
H^+) \cap H^1 \subseteq \cnh(\G_{01} \cap H^0) + \cvh(\G_{01} \cap
H^1).
\end{equation}
Then applying Lemma \ref{lem:reccone} with $S := \G_{01} \cap H^1$ and
$\rec(S) = \G_{01} \cap H^0$, we see that $\cnh(\G_{01}\cap H^+) \cap
H^1 \subseteq \cvh(\G_{01} \cap H^1)$, which proves the lemma.

To prove the claim (\ref{equ:local10}), let $x \in \cnh(\G_{01} \cap H^+)
\cap H^1$. Then
\[
    h^T x =1 \ \ \ \ \mbox{and} \ \ \ \ x = \sum_k \lambda_k x_k, \ \ \
    \ x_k \in \G_{01} \cap H^+, \ \ \ \ \lambda_k > 0,
\]
which may further be separated as
\[
  x = \underbrace{\sum_{k \, : \, h^T x_k > 0} \lambda_k x_k}_{:=y}
  + \underbrace{\sum_{k \, : \, h^T x_k = 0} \lambda_k x_k}_{:=r}
  = y + r.
\]
Note that $r \in \cnh(\G_{01} \cap H^0)$, and so it sufficies to show $y
\in \cvh(\G_{01} \cap H^1)$. Rewrite $y$ as
\[
 y = \sum_{k \, : \, h^T x_k > 0} \lambda_k x_k
    = \sum_{k \, : \, h^T x_k > 0} \underbrace{( \lambda_k \cdot h^T x_k)}_{:=\tilde{\lambda}_k} \underbrace{(x_k/h^T x_k)}_{:=\tilde{x}_k}
   \  =: \sum_{k \, : \, h^T x_k > 0} \tilde\lambda_k \tilde x_k.
\]
By construction, each $\tilde x_k \in \G_{01} \cap H^1$. Moreover, each
$\tilde \lambda_k$ is positive and
\[
  \sum_{k \, : \, h^T x_k > 0} \tilde \lambda_k =
  \sum_{k \, : \, h^T x_k > 0} \lambda_k \cdot h^T x_k =
  h^T y =
  h^T (x - r) = 1 - 0 = 1, 
\]
since $x \in H^1$.  So $y \in \cvh(\G_{01} \cap H^1)$. 
\qed \end{proof}

\begin{proposition} \label{pro:H1}
Suppose Conditions \ref{ass:one_neg_eval}--\ref{ass:hyperplane} hold. Then 
$
\F_0^+ \cap \F_1 \cap H^1   \subseteq
\F_0^+ \cap \F_s^+ \cap H^1 \subseteq
\cvh(\F_0^+ \cap \F_1 \cap H^1).
$
\end{proposition}

\begin{proof}
Define $\G_{01} := \F_0^+ \cap \F_1$ and $\G_{0s} := \F_0^+ \cap
\F_s^+$. Lemma \ref{lem:halfspace1} and Corollary \ref{cor:halfspace2}
imply $\G_{01} \cap H^+ \subset \G_{0s} \cap H^+ \subseteq \cnh(\G_{01})
\cap H^+$. Then Lemma \ref{lem:coneCross-section} implies the result.
\qed \end{proof}

As with Lemma \ref{lem:cross-section}, we have stated Lemma
\ref{lem:coneCross-section} in terms of general cones, extending beyond
just SOCr cones. In particular, in future research, these results may
allow the derivation of conic and convex hulls for the intersects with
more general cones.


\section{Two-term disjunctions on the second-order cone} \label{sec:twoterm}

In this section (specifically Sections
\ref{sec:d1=d2=0}--\ref{sec:d1>d2}), we consider the intersection of the
canonical second-order cone
\[
  \K := \{ x : \|\tilde{x}\|\leq x_n \}, \ \ \ \
  \text{where } \tilde x = (x_1; \ldots; x_{n-1}),
\]
and a two-term linear disjunction defined by $c_1^Tx\geq d_1
\,\vee\, c_2^Tx\geq d_2$. Without loss of generality, we take
$d_1,d_2\in\{0,\pm1\}$ with $d_1 \ge d_2$, and we work with the following
condition:

\begin{condition}\label{ass:non-intersecting}
The disjunctive sets $\K_1 := \K \cap \{ x : c_1^T x \ge d_1 \}$ and
$\K_2 := \K \cap \{ x : c_2^T x \ge d_2 \}$ are non-intersecting except
possibly on their boundaries, e.g.,
\[
  \K_1 \cap \K_2
\subseteq
\left\{x \in \K : \begin{array}{l} c_1^Tx = d_1 \\ c_2^Tx =  d_2 \end{array} \right\}.
\]
\end{condition}

\noindent This condition ensures that, on $\K$, the disjunction
$c_1^T x \ge d_1 \, \vee \, c_2^T x \ge d_2$ is equivalent to the
quadratic inequality $(c_1^T x - d_1)(c_2^Tx - d_2) \le 0$. Condition
\ref{ass:non-intersecting} is satisfied, for example, when the
disjunction is a proper split, i.e., $c_1 \parallel c_2$ with $c_1^T c_2
< 0$, $\K_1\cup\K_2 \neq\K$, and $d_1 = d_2$. (In this case of a split disjunction, if $d_1 \ne d_2$, then it
can be shown that the closed conic hull of $\K_1\cup\K_2$ is just $\K$.)

Because $d_1,d_2 \in \{0,\pm 1\}$ with $d_1 \ge d_2$, we can break our
analysis into the following three cases with a total of six subcases:
\begin{itemize}
  \item[(a)] $d_1 = d_2 = 0$, covering subcase $(d_1,d_2) = (0,0)$;
  \item[(b)] $d_1 = d_2$ nonzero, covering subcases $(d_1,d_2) \in \{(-1,-1), (1,1)\}$;
  \item[(c)] $d_1 > d_2$, covering subcases $(d_1,d_2) \in \{(0,-1), (1,-1), (1,0) \}$.
\end{itemize}
Case (a) is the homogeneous case, in which we take $A_0 = J :=
\Diag(1,\ldots,1,-1)$ and $A_1 = c_1 c_2^T + c_2 c_1^T$ to match our set
of interest $\K \cap \F_1$. Note that $\K = \F_0^+$ in this case. For
the non-homogeneous cases (b) and (c), we can homogenize via $y = {x
\choose x_{n+1}}$ with $h^T y = x_{n+1} = 1$. Defining
\[
A_0 := \begin{pmatrix} J & 0 \\ 0 & 0 \end{pmatrix}, \ \ \
A_1 :=
\begin{pmatrix}
  c_1 c_2^T + c_2 c_1^T & -d_2 c_1 - d_1 c_2  \\
 -d_2 c_1^T - d_1 c_2^T  & 2 d_1 d_2
\end{pmatrix},
\]
we then wish to examine $\F_0^+ \cap \F_1 \cap H^1$.

In fact, by the results in \cite[Section 5.2]{KY14}, case (c) implies
that $\ccnh(\F_0^+ \cap \F_1)$ cannot in general
be captured by two conic inequalities, making it unlikely
that our desired equality $\ccvh(\F_0^+ \cap \F_1 \cap H^1) = \F_0^+
\cap \F_s^+ \cap H^1$ will hold in general. So we will focus on cases (a) and
(b). Nevertheless, we include some comments on case (c) in Section
\ref{sec:d1>d2}.%

Later on, in Section \ref{sec:ass:non-intersecting}, we will also
revisit Condition \ref{ass:non-intersecting} to show that it
is unnecessary in some sense. Precisely, even when Condition
\ref{ass:non-intersecting} does not hold, we can derive a related convex
valid inequality, which, together with $\F_0^+$, gives the complete
convex hull description. This inequality precisely matches the one
already described in \cite{KY14}, but it does not have an SOC form.

In contrast to Sections \ref{sec:d1=d2=0}--\ref{sec:d1>d2}, Section
\ref{sec:conicsections} examines two-term disjunctions on conic sections
of $\K$, i.e., intersections of $\K$ with a hyperplane.

\subsection{The case (a) of $d_1 = d_2 = 0$}\label{sec:d1=d2=0}

As discussed above, we have $A_0 := J$ and $A_1 := c_1 c_2^T + c_2
c_1^T$. If either $c_i \in \K$, then the corresponding side of the
disjunction $\K_i$ simply equals $\K$, so the conic hull is $\K$. In
addition, if either $c_i \in \myint(-\K)$, then $\K_i = \{0\}$, so
the conic hull equals the other $\K_j$. Hence, we assume both $c_i
\not\in \K \cup \myint(-\K)$, i.e., $\|\tilde c_i\| \ge |c_{i,n}|$,
where $c_i = {\tilde c_i \choose c_{i,n}}$. Since the example in Section~4 of the Online Supplement violates Condition \ref{ass:As_null} with $\|
\tilde c_2 \| = |c_{2,n}|$, we further assume that both $\|\tilde c_i\|
> |c_{i,n}|$.

Conditions \ref{ass:one_neg_eval} and \ref{ass:A0_apex}(i) are easily
verified. In particular, $s > 0$. Condition \ref{ass:interior}
describes the full-dimensional case of interest. It remains to verify
Condition \ref{ass:As_null}. (Note that Condition \ref{ass:As_null} is
only relevant when $s < 1$ and that Condition \ref{ass:hyperplane} is
not of interest in this homogeneous case.) So suppose $s < 1$, and given
nonzero $z \in \Null(A_s)$, we will show
\[
  z^T A_1 z = 2 (c_1^T z)(c_2^T z) < 0,
\]
verifying Condition \ref{ass:As_null}. We already know from Lemma
\ref{lem:As_null} that $z^T A_1 z \le 0$. So it remains to show that
both $c_1^T z$ and $c_2^T z$ are nonzero.

Since $z \in \Null(A_s)$, we know $\left( \tfrac{1-s}{s} \right) A_0 z =
-A_1 z$, i.e.,
\be
\left( \tfrac{1-s}{s} \right) {\tilde{z} \choose -z_n}  = - c_1 (c_2^Tz) - c_2 (c_1^Tz).
\ee{eq:eq0}
Note that $c_1^Tz={\tilde{c}_1 \choose -c_{1,n}}^T {\tilde{z} \choose -z_n}$, so multiplying both sides of equation \eqref{eq:eq0} with ${\tilde{c}_1 \choose -c_{1,n}}^T$ and rearranging terms, we obtain
\[
\left[ \tfrac{1-s}{s} + \tilde{c}_1^T\tilde{c}_2 - c_{1,n} c_{2,n} \right] (c_1^Tz) = \left(c_{1,n}^2 - \|\tilde{c}_1\|_2^2 \right) (c_2^T z).
\]
Similarly, using ${\tilde{c}_2 \choose -c_{2,n}}^T$, we obtain:
\[
\left[ \tfrac{1-s}{s} + \tilde{c}_1^T\tilde{c}_2 - c_{1,n} c_{2,n} \right] (c_2^Tz) = \left(c_{2,n}^2 - \|\tilde{c}_2\|_2^2 \right) (c_1^T z).
\]
The inequalities $\| \tilde c_1 \| > |c_{1,n}|$ and $\| \tilde c_2 \|
> |c_{2,n}|$ thus imply $c_1^T z \ne 0 \Leftrightarrow c_2^T z \ne 0$.
Moreover, $c_1^T z$ and $c_2^T z$ cannot both be 0; otherwise, $z$ would
be 0 by (\ref{eq:eq0}).

Note that \cite{K-KY14,KY14} give an infinite family of valid
inequalities in this setup but do not prove the sufficiency of a
single inequality from this family. In this case, the sufficiency
proof for a single inequality from this family is given recently in
\cite{YC14}. None of the other papers \cite{AJ2013,DDV2011,MKV} are
relevant here because they consider only split disjunctions, not general
two-term disjunctions. Because of the boundedness assumption used in
\cite{BGPRT}, \cite{BGPRT} is not applicable here either. Similar to the
example in Section~1 of the Online Supplement, as long as the disjunction
can be viewed as removing a convex set, we can try to apply \cite{BM} to
this case by considering the SOC as the epigraph of the norm $\|\tilde
x \|$. However, the authors' special conditions for polynomial-time
separability such as differentiability or growth rate are not satisfied;
see Theorem IV therein.

\subsection{The case (b) of nonzero $d_1 = d_2$}\label{sec:d1=d2}

In  \cite{KY14}, it was shown that $c_1 - c_2 \in \pm \K$ implies one
of the sets $\K_i$ defining the disjunction is contained in the other
$\K_j$, and thus the desired closed convex hull trivially equals $\K_j$.
So we assume $c_1-c_2\not\in\pm\K$, i.e., $\|\tilde{c}_1-\tilde{c}_2\|^2
> (c_{1,n}-c_{2,n})^2$, where $c_i = {\tilde c_i \choose c_{i,n}}$.

Defining $\sigma = d_1 = d_2$, we have
\[
A_0 := \begin{pmatrix} J & 0 \\ 0 & 0 \end{pmatrix}, \ \ \
A_1 :=
\begin{pmatrix}
  c_1 c_2^T + c_2 c_1^T & -\sigma(c_1 + c_2)  \\
 -\sigma (c_1+c_2)^T   & 2
\end{pmatrix}.
\]
Conditions \ref{ass:one_neg_eval} and \ref{ass:A0_apex}(ii) are
easily verified, and Condition \ref{ass:interior} describes the
full-dimensional case of interest. It remains to verify Conditions
\ref{ass:As_null} and \ref{ass:hyperplane}. So assume $s < 1$, and note
$s > 0$ due to Condition \ref{ass:A0_apex}(ii).

For any $z^+ \in \R^{n+1}$, write $z^+={z \choose z_{n+1}}$ and
$z={\tilde{z} \choose z_n} \in \R^n$. Suppose $z^+ \ne 0$. Then
\begin{align*}
z^+ \in \Null(A_s)
\ \ \ \Longleftrightarrow \ \ \
\left( \tfrac{1-s}{s} \right) A_0 z^+ &= - A_1 z^+ \\
\ \ \ \Longleftrightarrow \ \ \
\left( \tfrac{1-s}{s} \right) A_0 z^+ &=
- \textstyle{c_1 \choose -\sigma}\textstyle{c_2 \choose -\sigma}^T z^+ - \textstyle{c_2 \choose  -\sigma} \textstyle{c_1 \choose -\sigma}^T z^+ \\
 &=: \alpha \textstyle{c_1 \choose -\sigma} + \beta \textstyle{c_2 \choose -\sigma}.
\end{align*}
Since the last component of $A_0 z^+$ is zero, we must have $\beta =
-\alpha$. We claim $\alpha \ne 0$. Assume for contradiction that
$\alpha=0$. Then $z=0$, but $z_{n+1}\neq 0$ as $z^+$ is nonzero. On the
other hand, because $z^+\in\Null(A_s)$, Lemma \ref{lem:As_null} implies $0\geq (z^+)^T A_1 z^+ =
2z_{n+1}^2$, a contradiction. So indeed $\alpha \ne 0$.

Because $z^+ \in \Null(A_s)$ and $s \in (0,1)$, the equation
\[
  0 = (z^+)^T A_s z^+ = (1-s)(z^+)^T A_0 z^+ + s (z^+)^T A_1 z^+ ,
\]
implies Condition \ref{ass:As_null} holds if and only if $(z^+)^T A_0
z^+ > 0$. From the previous paragraph, we have $\left( \tfrac{1-s}{s}
\right) A_0 z^+ = \alpha {c1 - c2 \choose 0}$ with $\alpha \ne 0$. Then
\begin{align*}
   \left( \tfrac{1-s}{s} \right) (z^+)^T A_0 z^+ &=
  \begin{pmatrix}
    \alpha (\tilde c_1 - \tilde c_2) \\
    -\alpha (c_{1,n} - c_{2,n}) \\
    z_{n+1} 
  \end{pmatrix}^T
  \begin{pmatrix}
    \alpha (\tilde c_1 - \tilde c_2) \\
    \alpha (c_{1,n} - c_{2,n}) \\
    0 
  \end{pmatrix} \\
  &= \alpha^2 \left( \| \tilde c_1 - \tilde c_2 \|^2 - (c_{1,n} - c_{2,n})^2 \right) > 0,
\end{align*}
as desired.

However, it seems difficult to verify Condition \ref{ass:hyperplane}
generally. For example, consider its second part $\F_0^+ \cap
\F_s^+ \cap H^0 \subseteq \F_1$. In the current context, we have $\F_0^+
\cap H^0=\K\times\{0\}$, and it is unclear if its intersection with
$\F_s^+$ would be contained in $\F_1$. Letting ${\hat h \choose 0 } \in
\F_s^+$ with $\hat h \in \K$, we would have to check the following:
\[
  0 \ge {\hat h \choose 0 }^T A_s {\hat h \choose 0 }
  = (1-s) \, \hat h^T J \hat h + 2s \,  (c_1^T \hat h)(c_2^T \hat h)
   \ \ \ \Longrightarrow \ \ \ 
    {\hat h \choose 0} \in \F_1.
\]
If $\hat h$ were in the interior of $\K$, then $\hat h^T J \hat h < 0$
could still allow $(c_1^T \hat h)(c_2^T \hat h) > 0$, so that ${\hat
h \choose 0} \in \F_1$ would not be achieved. So it seems Condition
\ref{ass:hyperplane} will hold under additional conditions only.

One such set of conditions ensuring Condition \ref{ass:hyperplane} is as follows: there exists $\beta_1,\beta_2
\ge 0$ such that $\beta_1 c_1 + c_2 \in -\K$ and $\beta_2 c_1 + c_2
\in \K$. These hold, for example, for split disjunctions, i.e.,
when $c_2$ is a negative multiple of $c_1$. To prove Condition
\ref{ass:hyperplane}, take $\hat h \in \K$. Then $c_1^T \hat h \ge 0$ implies
\[
  c_2^T \hat h = -\beta_1 c_1^T \hat h + (\beta_1 c_1 + c_2)^T \hat h \le 0 + 0 = 0,
\]
and similarly $c_1^T \hat h \le 0$ implies $c_2^T \hat h \ge 0$. Then overall
$\hat h \in \K$ implies $(c_1^T \hat h)(c_2^T \hat h) \le 0$. In the
context of the previous paragraph, this ensures $\F_0^+ \cap \F_s^+ \cap
H^0 \subseteq \F_0^+ \cap H^0 \subseteq \F_1$, thus verifying Condition
\ref{ass:hyperplane}.

Note that \cite{K-KY14,KY14} cover this case. In the case of split
disjunctions with $d_1=d_2=1$, these results are also presented in
\cite{AJ2013,MKV}. Whenever the boundedness assumption of \cite{BGPRT}
is satisfied, one can use their result as well, but the papers
\cite{DDV2011,YC14} are not relevant here. Similar to the previous
subsection, \cite{BM} is limited in its application to this case.

\subsection{Revisiting Condition \ref{ass:non-intersecting}}\label{sec:ass:non-intersecting}

For the cases $d_1 = d_2$ of Sections \ref{sec:d1=d2=0}
and \ref{sec:d1=d2}, we know that $\F_0^+ \cap \F_s^+$
is a valid convex relaxation of $\F_0^+ \cap \F_1$ under
Conditions \ref{ass:one_neg_eval}--\ref{ass:A0_apex} and
\ref{ass:non-intersecting}. The same holds for the cross-sections:
$\F_0^+ \cap \F_s^+ \cap H^1$ is a relaxation of $\F_0^+ \cap \F_1 \cap
H^1$. Because Condition \ref{ass:A0_apex}(i) is verified in the case
of $d_1=d_2=0$ and Condition \ref{ass:A0_apex}(ii) is verified in the
case of nonzero $d_1=d_2$, we have $s > 0$. However, when Condition
\ref{ass:non-intersecting} is violated, it may be possible that $\F_s^+$
is invalid for points simultaneously satisfying both sides of the
disjunction, i.e., points $x$ with $c_1^T x \ge d_1$ and $c_2^T x \ge
d_2$. This is because such points can violate the quadratic $(c_1^T x
- d_1)(c_2^T x - d_2) \le 0$ from which $\F_s^+$ is derived. In such
cases, the set $\F_s^+$ should be relaxed somehow.

Recall that, by definition, $\F_s^+ = \{ x : x^T A_s x \le 0,\  b_s^T x
\ge 0 \}$. Let us examine the inequality $x^T A_s x \le 0$, which can be
rewritten as
\bse
&& 0\geq (1-s)\, x^T J x + 2s \,  (c_1^T x-d_1)(c_2^T x-d_2) \\
&\Longleftrightarrow & 0\geq  2(1-s)\, x^T J x + s \,
\left( [(c_1^T x-d_1)+(c_2^T x-d_2)]^2 -   [(c_1^T x-d_1)-(c_2^T x-d_2)]^2 \right) \\
&\Longleftrightarrow & s\   [(c_1-c_2)^T x-(d_1-d_2)]^2 - 2(1-s)\, x^T J x \ \geq\  s\  [(c_1+c_2)^T x-(d_1+d_2)]^2.
\ese
Note that the left hand-side of the third inequality is nonnegative for
any $x\in\K$ since $x^T J x \le 0$. Therefore, $x\in\K$ implies $x^T A_s
x \le 0$ is equivalent to
\be
\sqrt{ \left[(c_1-c_2)^T x-(d_1-d_2)\right]^2 - 2\left( \tfrac{1-s}{s}\right) x^T J x} \ \geq \  |(c_1+c_2)^T x-(d_1+d_2)|.
\ee{equ:local03}
An immediate relaxation of (\ref{equ:local03}) is
\be
\sqrt{ \left[(c_1-c_2)^T x-(d_1-d_2)\right]^2 - 2\left( \tfrac{1-s}{s}\right) x^T J x} \ \geq \  (d_1+d_2) - (c_1+c_2)^T x
\ee{equ:local04}
since $|(c_1+c_2)^T x-(d_1+d_2)| \ge (d_1+d_2) - (c_1+c_2)^Tx$.
Note also that (\ref{equ:local04}) is clearly valid for any $x$
satisfying $c_1^T x \ge d_1$ and $c_2^T x \ge d_2$ since the two
sides of the inequality have different signs in this case. In
total, the set
\[
  {\cal G}_s^+ := \{ x : (\ref{equ:local04}) \text{ holds},\ b_s^T x \ge 0 \}
\]
is a valid relaxation when Condition \ref{ass:non-intersecting}
does not hold. Although not obvious, it follows from \cite{KY14}
that (\ref{equ:local04}) is a convex inequality. In that paper,
(\ref{equ:local04}) was encountered from a different viewpoint, and its
convexity was established directly, even though it does not admit an SOC
representation. So in fact ${\cal G}_s^+$ is convex.

Now let us assume that Condition \ref{ass:As_null}
holds as well so that $\F_s^+$ captures the conic hull
of the intersection of $\F_0^+$ and $(c_1^T x - d_1)(c_2^T x - d_2) \le
0$.
We claim that $\F_0^+ \cap {\cal G}_s^+$ captures the conic hull when Condition
\ref{ass:non-intersecting} does not hold. (A similar claim will also
hold when Condition \ref{ass:hyperplane} holds for the further intersection
with $H^1$.) 
So let $\hat x \in \F_0^+ \cap {\cal G}_s^+$ be given.
If (\ref{equ:local03}) happens to hold also,
then $\hat x^T A_s \hat x \le 0 \Rightarrow \hat x \in \F_s^+$. Then
$\hat{x}$ is already in the closed convex hull given by $(c_1^T x - d_1)(c_2^T x
- d_2) \le 0$ by assumption. 
On the other hand, if (\ref{equ:local03}) does
not hold, then it must be that $(c_1+c_2)^T\hat{x} > d_1+d_2$. So
either $c_1^T\hat{x}>d_1$ or $c_2^T\hat{x}>d_2$. Whichever the case,
$\hat{x}$ satisfies the disjunction. Therefore $\hat{x}$ is in the
closed convex hull, which gives the desired conclusion. 

We remark that, despite their different forms, (\ref{equ:local04}) and
the inequality defining $\F_s^+$ both originate from $x^T A_s x \le 0$
and match precisely on the boundary of $\cnh(\F_0^+ \cap \F_1 )\setminus
(\F_0^+ \cap \F_1)$, e.g., the points added due to the convexification
process. Moreover, (\ref{equ:local04}) can be interpreted as adding all
of the recessive directions $\{d \in \K : c_1^Td\geq 0,\, c_2^Td\geq
0\}$ of the disjunction to the set $\F_0^+\cap\F_s^+$. Finally, 
the analysis in \cite{KY14} shows in addition that the linear inequality $b_s^T x \ge 0$
is in fact redundant for ${\cal G}_s^+$.

Note that \cite{K-KY14,KY14} cover this case. Because the resulting
convex hull is not conic representable \cite{BGPRT} is not applicable
in this case. The papers \cite{DDV2011,YC14} are not relevant here and
none of the other papers \cite{AJ2013,MKV} cover this case because they
focus on split disjunctions only. As in the previous two subsections,
\cite{BM} is limited in its application.

\subsection{The case (c) of $d_1 > d_2$} \label{sec:d1>d2}

As mentioned above, the results of \cite{KY14} ensure that $\ccnh(\F_0^+
\cap \F_1)$ requires more than two conic inequalities, making it highly
likely that the closed convex hull of $\F_0^+ \cap \F_1 \cap H^1$
requires more than two also. In other words, our theory would not apply
in this case in general. So we ask: which conditions are violated
in this case? 

Let us first consider when $d_1 d_2 = 0$, which covers two subcases.
Then
\[
A_0 := \begin{pmatrix} J & 0 \\ 0 & 0 \end{pmatrix}, \ \ \
A_1 :=
\begin{pmatrix}
  c_1 c_2^T + c_2 c_1^T & -d_2 c_1 - d_1 c_2  \\
 -d_2 c_1^T - d_1 c_2^T  & 0
\end{pmatrix},
\]
and it is clear that Condition \ref{ass:A0_apex} is not satisfied.

Now consider the remaining subcase when $(d_1,d_2) = (1,-1)$. Then
\[
A_0 := \begin{pmatrix} J & 0 \\ 0 & 0 \end{pmatrix}, \ \ \
A_1 :=
\begin{pmatrix}
  c_1 c_2^T + c_2 c_1^T & c_1 - c_2  \\
 c_1^T - c_2^T  & -2
\end{pmatrix}.
\]
Condition \ref{ass:one_neg_eval} holds, and Condition
\ref{ass:interior} is the full-dimensional case of interest. Condition
\ref{ass:A0_apex}(iii) holds as well, so $s=0$. Then Condition
\ref{ass:As_null} requires $v^T A_1 v < 0$, where $v = (0;\ldots;0;1)$,
which is true. On the other hand, Condition \ref{ass:hyperplane} might
fail. In fact, the example in Section~5 of the Online Supplement provides
just such an instance. This being said, the same stronger condition
discussed in Section \ref{sec:d1=d2} can be seen to imply Condition
\ref{ass:hyperplane}, that is, when there exists $\beta_1,\beta_2
\ge 0$ such that $\beta_1 c_1 + c_2 \in -\K$ and $\beta_2 c_1 + c_2
\in \K$. This covers the case of split disjunctions, for example.

Of course, even when all conditions do not hold, just Conditions
\ref{ass:one_neg_eval}-\ref{ass:A0_apex}, which hold when $d_1 d_2=-1$,
are enough to ensure the valid relaxations $\F_0^+ \cap \F_s^+$ and
$\F_0^+ \cap \F_s^+ \cap H^1$. However, these relaxations may not be
sufficient to describe the conic and convex hulls.

If necessary, another way to generate valid conic inequalities
when $d_1>d_2$ is as follows. Instead of the original disjunction,
consider the weakened disjunction $c_1^Tx\geq d_2 \,\vee\, c_2^Tx\geq
d_2$, where $d_2$ replaces $d_1$ in the first term. Clearly any
point satisfying the original disjunction will also satisfy the new
disjunction. Therefore any valid inequality for the new disjunction
will also be valid for the original one. In Sections \ref{sec:d1=d2=0}
and \ref{sec:d1=d2}, we have discussed the conditions under which
Conditions \ref{ass:one_neg_eval}-\ref{ass:hyperplane} are satisfied
when $d_1=d_2$. Even if the new disjunction violates Condition
\ref{ass:non-intersecting}, as long as the original disjunction
satisfies Condition \ref{ass:non-intersecting}, the resulting
inequalities from this approach will be valid.

Regarding the existing literature, the conclusions at the end of Section
\ref{sec:ass:non-intersecting} also apply here.

\subsection{Conic sections} \label{sec:conicsections}

Let $\rho_1^Tx\geq d_1 \,\vee\, \rho_2^Tx\geq d_2$ be a disjunction on a
cross-section $\K \cap H^1$ of the second-order cone, where $H^1 = \{ x
: h^T x = 1 \}$. We work with an analogous of Condition
\ref{ass:non-intersecting}:

\begin{condition}\label{ass:non-intersecting-H1}
The disjunctive sets $\K_1 := \K \cap H^1 \cap \{ x : \rho_1^T x \ge
d_1 \}$ and $\K_2 := \K \cap H^1 \cap \{ x : \rho_2^T x \ge d_2 \}$ are
non-intersecting except possibly on their boundaries, e.g.,
\[
  \K_1 \cap \K_2
\subseteq
\left\{x \in \K \cap H^1 :
  \begin{array}{l} \rho_1^Tx = d_1 \\ \rho_2^Tx =  d_2 \end{array} \right\}.
\]
\end{condition}

\noindent We would like to characterize the convex hull of the
disjunction, which is the same as the convex hull of the disjunction
$(\rho_1 - d_1 h)^Tx \geq 0 \,\vee\, (\rho_2 - d_2 h)^Tx \geq 0$ on $\K
\cap H^1$. Defining $c_1:=\rho_1 - d_1 h$, $c_2:=\rho_2 - d_2 h$, $A_0
:= J$, and $A_1 := c_1 c_2 ^T + c_2 c_1^T$, our goal is to characterize
$\ccvh(\K \cap \F_1 \cap H^1)$. This is quite similar to the analysis
in Section \ref{sec:d1=d2=0} except that here we also must verify
Condition \ref{ass:hyperplane}.

Conditions \ref{ass:one_neg_eval} and \ref{ass:A0_apex}(i) are
easily verified, and Condition \ref{ass:interior} describes the
full-dimensional case of interest. Following the development in Section
\ref{sec:d1=d2=0}, we can verify Condition \ref{ass:As_null} when
$\|\tilde{\rho}_1 - d_1 \tilde{h}\|_2 > | \rho_{1,n} - d_1h_n |$ and
$\|\tilde{\rho}_2 - d_2 \tilde{h}\|_2 > | \rho_{2,n} - d_2 h_n |$,
and otherwise the convex hull is easy to determine. For Condition
\ref{ass:hyperplane}, we consider the cases of ellipsoids, paraboloids,
and hyperboloids separately.

Ellipsoids are characterized by $h \in \myint(\K)$, and so $\K \cap
H^0 = \{0\}$. Thus $\K \cap \F_s^+ \cap H^0 = \{ 0\} \subseteq \F_1$
easily verifying Condition \ref{ass:hyperplane}.
On the other hand, paraboloids are characterized by $0 \ne h \in \bd(\K)$,
and in this case, $\K \cap H^0 = \cone\{ \hat h \}$, where $\hat
h := -Jh = {-\tilde{h} \choose h_n}$. Thus, to verify Condition
\ref{ass:hyperplane}, it suffices to show $\hat h \in \F_s^+
\, \Rightarrow \, \hat h \in \F_1$. Indeed $\hat h \in \F_s^+$ implies
\[
  0 \ge \hat h^T A_s \hat h
  = (1-s) \, \hat h^T J \hat h + s \, \hat h^T A_1 \hat h
  = s \, \hat h^T A_1 \hat h
\]
because $h \in \bd(\K)$ ensures $\hat h^T J \hat h = 0$. So $\hat h \in \F_1$.

It remains only to verify Condition \ref{ass:hyperplane} for
hyperboloids, which are characterized by $h\notin\pm\K$, i.e., $h =
{\tilde h \choose h_n}$ satisfies $\|\tilde h\| > |h_n|$. However, it
seems difficult to verify Condition \ref{ass:hyperplane} generally.
Still, we note that $\hat h \in H^0$ implies
\[
  \hat h^T A_1 \hat h = 2 (c_1^T \hat h)(c_2^T \hat h) =
  2(\rho_1^T \hat h - d_1 h^T \hat h) (\rho_2^T \hat h - d_2 h^T \hat h) =
  2 (\rho_1^T \hat h)(\rho_2^T \hat h).
\]
Then Condition \ref{ass:hyperplane} would hold, for example, when
$\rho_1$ and $\rho_2$ satisfy the following, which is identical to
conditions discussed in Sections \ref{sec:d1=d2} and \ref{sec:d1>d2}:
there exists $\beta_1,\beta_2 \ge 0$ such that $\beta_1 \rho_1 + \rho_2
\in -\K$ and $\beta_2 \rho_1 + \rho_2 \in \K$. This covers the case of
split disjunctions, for example.

We remark that our analysis in this subsection covers all of the various
cases of split disjunctions found in \cite{MKV} and more. In particular,
we handle ellipsoids and paraboloids for all possible general two-term
disjunctions (including the non-disjoint ones). On the other hand,
the cases we can cover for hyperboloids is a subset of those recently
given in \cite{YC14}. Note that \cite{DDV2011} covers only split
disjunctions on ellipsoids. \cite{BGPRT} covers two-term disjunctions on
ellipsoids and certain specific two-term disjunctions on paraboloids and
hyperboloids satisfying their disjointness and boundedness assumptions.
None of the papers \cite{AJ2013,K-KY14,KY14} are relevant here. Finally,
when the disjunction correspond to the deletion of a convex set, the
paper \cite{BM} applies to the cases for ellipsoids and paraboloids
because those sets can be viewed as epigraphs of strictly convex
quadratics.

\section{General Quadratics with Conic Sections} \label{sec:general}

In this section, we examine the case of (nearly) general quadratics
intersected with conic sections of the SOC. For simplicity of
presentation, we will employ affine transformations of the sets $\F_0^+
\cap \F_1 \cap H^1$ of interest. It is clear that our theory is not
affected by affine transformations.

\subsection{Ellipsoids}

Consider the set
\[
  \left\{ y \in \R^n \ :
    \begin{array}{cc}
      y^T y \le 1 \\
      y^T Q y + 2 \, g^T y + f \le 0
    \end{array}
    \right\},
\]
where $\lambda_{\min}[Q] < 0$. Note that if $\lambda_{\min}[Q] \ge 0$, then the set is already convex. Allowing an affine transformation, this set 
models the intersection of any ellipsoid with a general quadratic
inequality. We can model this set in our framework by
homogenizing $x = {y \choose x_{n+1}}$ and taking
\[
  A_0 := \begin{pmatrix} I & 0 \\ 0^T & -1 \end{pmatrix}, \ \ \
  A_1 := \begin{pmatrix} Q & g \\ g^T & f \end{pmatrix}, \ \ \
  H^1 := \{ x : x_{n+1} = 1 \}.
\]
We would like to compute $\ccvh(\F_0^+ \cap \F_1 \cap H^1)$.

Conditions \ref{ass:one_neg_eval} and \ref{ass:A0_apex}(i) are clear,
and Condition \ref{ass:interior} describes the full-dimensional case
of interest. When $s < 1$, Condition \ref{ass:hyperplane} is satisfied
because, in this case, $\F_0^+ \cap H^0 = \{0\}$ making the containment
$\F_0^+ \cap \F_s^+ \cap H^0 \subseteq \F_1$ trivial. 
In Sections \ref{sec:subcase1} and \ref{sec:subcase2} below, we
break the analysis
of verifying Condition \ref{ass:As_null} 
into two subcases that we are able to handle: (i)
when $\lambda_{\min}[Q]$ has multiplicity $k \ge 2$; and (ii) when
$\lambda_{\min}[Q] \le f$ and $g=0$.

Subcase (i) covers, for example, the situation of deleting the interior
of an arbitrary ball from the unit ball. Indeed, consider
\[
  \left\{ x \in \R^n \ :
    \begin{array}{cc}
      x^T x \le 1 \\
      (x-c)^T(x-c) \ge r^2
    \end{array}
    \right\},
\]
where $c \in \R^n$ and $r > 0$ are the center and radius of the ball
to be deleted. Then case (i) holds with $(Q,g,f) = (-I,c,r^2 - c^T
c)$. On the other hand, subcase (ii) can handle, for example, the
deletion of the interior of an arbitrary ellipsoid from the unit
ball---as long as that ellipsoid shares the origin as its center. In
other words, the portion to delete is defined by $x^T E x < r^2$, for
some $E \succ 0$ and $r > 0$, and we take $(Q,g,f) = (-E,0,r^2)$. Note
that $\lambda_{\min}[Q] \le -f \Leftrightarrow \lambda_{\max}[E] \ge
r^2$, which occurs if and only if the deleted ellipsoid contains a
point on the boundary of the unit ball. This is the most interesting
case because, if the deleted ellipsoid were either completely inside
or outside the unit ball, then the convex hull would simply be the
unit ball itself. The subcase (ii) was also studied in Corollary 9
of \cite{MKV} and in \cite{BM}. Moreover, none of the other papers
\cite{AJ2013,BGPRT,DDV2011,K-KY14,KY14,YC14} can handle this case.

\subsubsection{When $\lambda_{\min}[Q]$ has multiplicity $k \ge 2$} \label{sec:subcase1}

Define $B_t := (1-t)I + tQ$ to be the top-left $n \times n$ corner
of $A_t$. Since $\lambda_{\min}[B_1] < 0$ with multiplicity $k \ge 2$,
there exists $r \in (0,1)$ such that: (i) $B_r \succeq 0$; (ii)
$\lambda_{\min}[B_r] = 0$ with multiplicity $k$; (ii) $B_t \succ
0$ for all $t < r$. We claim that $s = r$ as a consequence of the
interlacing of eigenvalues with respect to $A_t$ and $B_t$. Indeed, let
$\lambda^t_{n+1} := \lambda_{\min}[A_t]$ and $\lambda^t_n$ denote the
two smallest eigenvalues of $A_t$, and let $\rho^t_n$ and $\rho^t_{n-1}$
denote the analogous eigenvalues of $B_t$. It is well known that
\[
  \lambda^t_{n+1} \ \le \ \rho^t_n \ \le \ \lambda^t_n \ \le \ \rho^t_{n-1}.
\]
When $t < r$, we have $\lambda^t_{n+1} < 0 < \rho^t_n \le \lambda^t_n$,
and when $t = r$, we have $\lambda^r_{n+1} < 0 \le \lambda^r_n \le 0$,
which proves $s=r$.

Since $\dim(\Null(B_s)) = k \ge 2$ and $\dim(\myspan\{g\}^\perp) = n-1$,
there exists $0 \ne z \in \Null(B_s)$ such that $g^T z = 0$. We can show
that ${z \choose 0} \in \Null(A_s)$:
\[
  A_s {z \choose 0}
  =
  \begin{pmatrix}
    B_s & s\, g \\ s \, g^T & (1-s)(-1) + sf
  \end{pmatrix}
  {z \choose 0}
  =
  {B_s z \choose s \, g^T z} = {0 \choose 0}.
\]
Moreover, ${z \choose 0}^T A_1 {z \choose 0} = z^T B_1 z = z^T Q z <
0$ because $z \in \Null(B_s)$ if and only if $z$ is a eigenvector of
$B_1 = Q$ corresponding to $\lambda_{\min}[Q]$. This verifies Condition
\ref{ass:As_null}.

\subsubsection{When $\lambda_{\min}[Q] \le -f$ and $g = 0$} \label{sec:subcase2}

The argument is similar to the preceding subcase in Section
\ref{sec:subcase1}. Note that
\[
  A_t = \begin{pmatrix} (1-t)I + tQ & 0 \\ 0 & (1-t)(-1) + t f \end{pmatrix}
   =: \begin{pmatrix} B_t & 0 \\ 0 & \beta_t \end{pmatrix}
\]
is block diagonal, so that the singularity of $A_t$ is determined
by the singularity of $B_t$ and $\beta_t$. $B_t$ is 
first singular when $t = 1/(1 - \lambda_{\min}[Q])$, while $\beta_t$ is first
singular when $t = 1/(1+f)$ (assuming $f > 0$; if not, then
$\beta_t$ is never singular). Then
\[
  \frac{1}{1 - \lambda_{\min}[Q]} \le \frac{1}{1+f}
  \ \ \
  \Longleftrightarrow
  \ \ \
  \lambda_{\min}[Q] \le -f,
\]
which holds by assumption. So $B_t$ is singular before $\beta_t$,
leading to $s = 1/(1 - \lambda_{\min}[Q])$. Let $0 \ne z \in \Null(B_s)$.
Then, we have $Qz=-{1-s\over s} z$, and thus, ${z \choose 0} \in \Null(A_s)$ with ${z \choose 0}^T A_1 {z \choose 0} = z^T B_1 z = z^T Q z < 0$. Condition \ref{ass:As_null} is
hence verified.

\subsection{The trust-region subproblem}

We show in this subsection that our methodology can be used to solve the
trust-region subproblem (TRS)
\begin{equation} \label{equ:trs}
  \min_{\tilde y \in \R^{n-1}}
  \left\{ \tilde y^T \tilde Q \tilde y + 2 \, \tilde g^T \tilde y :~
  \tilde y^T \tilde y \le 1 \right\},
\end{equation}
where $\lambda_{\min}[\tilde Q] < 0$. Without loss of generality, we
assume that $\tilde Q$ is diagonal with $\tilde Q_{(n-1)(n-1)} =
\lambda_{\min}[\tilde Q]$ after applying an orthogonal transformation
that does not change the feasible set.

Our intention is not necessarily to argue that the TRS should be solved numerically with our approach, although
this is an interesting question left as future work. Our goal is
to illustrate that the well-known problem (\ref{equ:trs}) can be
handled by our machinery. We also believe that the corresponding SOCP
formulation for the TRS as opposed to its usual
SDP formulation is independently interesting. Our transformations
to follow require simply two eigenvalue decompositions and the
resulting SOCP can be solved by interior point solvers very
efficiently. We note that none of the previous papers, in particular,
\cite{AJ2013,BGPRT,DDV2011,K-KY14,KY14,MKV,YC14} have given a
transformation of the TRS into an SOC optimization
problem before. We recently became aware that an SOC based reformulation of TRS was also given in Jeyakumar and Li \cite{JeyakumarLi2013}; our approach parallels their developments from a different, convexification based, perspective.

We first argue that (\ref{equ:trs}) is equivalent to a trust-region
subproblem
\begin{equation} \label{equ:trsalt}
  \min_{y \in \R^n} \left\{ y^T Q y + 2 \, g^T y :~ y^T y \le 1 \right\}
\end{equation}
in the $n$-dimensional variable $y := { \tilde y \choose y_n }$.
Indeed, define
\[
  Q := \begin{pmatrix} \tilde Q & 0 \\ 0^T & \lambda_{\min}[\tilde Q] \end{pmatrix},
  \ \ \ 
  g := {\tilde g \choose 0},
\]
and note that $\lambda_{\min}[Q]$ has multiplicity at least 2. The
following proposition shows that (\ref{equ:trsalt}) is equivalent to
(\ref{equ:trs}).

\begin{proposition}
There exists an optimal solution of (\ref{equ:trsalt}) with $y_n
= 0$. In particular, the optimal values of (\ref{equ:trs}) and
(\ref{equ:trsalt}) are equal.
\end{proposition}

\begin{proof}
Let $\bar y$ be an optimal solution of $(\ref{equ:trsalt})$. Then
$(\bar y_{n-1};\bar y_n)$ is an optimal solution of the two-dimensional
trust-region subproblem
\[
  \min_{y_{n-1},y_n} \left\{  
  |\lambda_{\min}[\tilde Q]| (-y_{n-1}^2 - y_n^2) + 2 \tilde g_{n-1} y_{n-1} : ~
    ~y_{n-1}^2 + y_n^2 \le \epsilon
   \right\}.
\]
where 
$\epsilon := 1 - (\bar y_1^2 + \cdots \bar y_{n-2}^2)$. Since we are minimizing a concave function over the ellipsoid, at least one optimal solution will be on the boundary of this set. In particular, whenever $\tilde g_{n-1}>0$,  the solution ${y_{n-1}\choose y_n} ={-\sqrt{\epsilon}\choose 0}$ is optimal, and when  $\tilde g_{n-1}\leq0$, the solution ${y_{n-1}\choose y_n} ={\sqrt{\epsilon}\choose 0}$ is optimal. Thus,  
this problem has at least one optimal solution with $y_n = 0$.
Hence, $\bar y_n$ can be taken as 0.
\qed \end{proof}

\noindent With the proposition in hand, we now focus on the solution of
(\ref{equ:trsalt}).

A typical approach to solve (\ref{equ:trsalt}) is to introduce
an auxiliary variable $x_{n+2}$ (where we reserve the variable
$x_{n+1}$ for later homogenization) and to recast the problem
as
\[
  \min \left\{ x_{n+2} :~ \begin{array}{ll} y^T y \le 1 \\ y^T Q y + 2 \, g^T y \le x_{n+2}
    \end{array} \right\}.
\]
If one can compute the closed convex hull of this feasible set, then
(\ref{equ:trsalt}) is solvable by simply minimizing $x_{n+2}$ over the
convex hull. We can represent this approach in our framework by taking
$x = (y;x_{n+1};x_{n+2})$, homogenizing via $x_{n+1} = 1$, and defining
\[
A_0 := \begin{pmatrix} I & 0 & 0 \\ 0^T & -1 & 0 \\ 0^T & 0 & 0 \end{pmatrix}, \ \ \
A_1 := \begin{pmatrix} Q & g & 0 \\ g^T & 0 & -\tfrac12  \\ 0^T & -\tfrac12 & 0 \end{pmatrix}, \ \ \
H^1 := \{ x \in \R^{n+2} : x_{n+1} = 1 \}.
\]
Clearly, Conditions \ref{ass:one_neg_eval} and \ref{ass:interior}
are satisfied. However, no part of Condition \ref{ass:A0_apex} is
satisfied. So we require a different approach.

Since $x = 0$ is feasible for (\ref{equ:trsalt}), its optimal
value is nonpositive. (In fact, it is negative since $Q$ has a
negative eigenvector, so that $x=0$ is not a local minimizer). Hence,
(\ref{equ:trsalt}) is equivalent to
\begin{equation} \label{equ:trsalt2}
  v := \min \left\{ x_{n+2}^2 :~ \begin{array}{ll} y^T y \le 1 \\ y^T Q y
  + 2 \, g^T y \le -x_{n+2}^2 \end{array} \right\},
\end{equation}
which can be solved in stages: first, minimize $x_{n+2}$ over the
feasible set of (\ref{equ:trsalt2}) (let $l$ be the minimal value); second,
separately maximize $x_{n+2}$ over the same (let $u$ be the maximal value); and
finally take $v = \min\{ -l^2, -u^2 \}$. If one can compute the closed
convex hull of (\ref{equ:trsalt2}), then $l$ and $u$ can be computed
easily.

To represent the feasible set of (\ref{equ:trsalt2}) in our framework,
we define $x = (y;x_{n+1};x_{n+2})$ and take
\[
A_0 := \begin{pmatrix} I & 0 & 0 \\ 0^T & -1 & 0 \\ 0^T & 0 & 0 \end{pmatrix}, \ \ \
A_1 := \begin{pmatrix} Q & g & 0 \\ g^T & 0 & 0  \\ 0^T & 0 & 1 \end{pmatrix}, \ \ \
H^1 := \{ x \in \R^{n+2} :~ x_{n+1} = 1 \}.
\]
Clearly, Conditions \ref{ass:one_neg_eval} and \ref{ass:interior} are
satisfied, and Condition \ref{ass:A0_apex}(ii) is now satisfied. For
Conditions \ref{ass:As_null} and \ref{ass:hyperplane}, we note that
$A_t$ has a block structure such that $s$ equals the smallest positive
$t$ such that
\[
  B_t := (1-t) \begin{pmatrix} I & 0 \\ 0 & -1 \end{pmatrix}
  + t \begin{pmatrix} Q & g \\ g^T & 0 \end{pmatrix}
\]
is singular. Using an argument similar to Section \ref{sec:subcase1}
and exploiting the fact that $\lambda_{\min}[Q]$ has multiplicity
at least 2, we can compute $s$ such that there exists $0 \ne z \in
\Null(B_s) \subseteq \R^{n+1}$ with $z^T B_1 z < 0$ and $z_{n+1}
= 0$. By appending an extra 0 entry, this $z$ can be
easily extended to $z \in \R^{n+2}$ with $z^T A_1 z < 0$ and $z \in
H^0$. This simultaneously verifies Conditions \ref{ass:As_null} and
\ref{ass:hyperplane}.

\subsection{Paraboloids}

Consider the set
\[
  \left\{ y = {\tilde y \choose y_n} \in \R^n \ :~
    \begin{array}{cc}
      \tilde y^T \tilde y \le y_n \\
      \tilde y^T \tilde Q \tilde y + 2 \, g^T y + f \le 0
    \end{array}
    \right\},
\]
where $\lambda := \lambda_{\min}[\tilde Q] < 0$ and $2 g_n \le
-\lambda$. After an affine transformation, this models the intersection
of a paraboloid with any quadratic inequality that is strictly linear
in $y_n$, i.e., no quadratic terms involve $y_n$. Note that if
$\lambda_{\min}[Q] \ge 0$, then the set is already convex. The reason
for the upper bound on $2 g_n$ will become evident shortly.

Writing $g := {\tilde g \choose g_n}$, we can model this
situation with $x = {y \choose x_{n+1}}$ and
\[
  A_0 := \begin{pmatrix} I & 0 & 0 \\ 0^T & 0 & -\tfrac12 \\ 0^T & -\tfrac12 & 0 \end{pmatrix}, \ \ \
  A_1 := \begin{pmatrix} \tilde Q & 0 & \tilde g \\ 0^T & 0 & g_n \\ \tilde g^T & g_n & f \end{pmatrix}, \ \ \
  H^1 := \{ x : x_{n+1} = 1 \},
\]
and we would like to compute $\ccvh(\F_0^+ \cap \F_1 \cap H^1)$.
Conditions \ref{ass:one_neg_eval} and \ref{ass:A0_apex}(i) are clear,
and Condition \ref{ass:interior} describes the full-dimensional case
of interest. So it remains to verify Conditions \ref{ass:As_null}
and \ref{ass:hyperplane}.

Define
\[
  B_t := \begin{pmatrix} (1-t)I + t \tilde Q & 0 \\ 0 & 0 \end{pmatrix}
\]
to be the top-left $n \times n$ corner of $A_t$, and define $r :=
1/(1-\lambda)$. Due to its structure, $B_t$ is positive semidefinite
for all $t \le r$. Moreover, $B_t$ has exactly one zero eigenvalue
for $t < r$, and $B_r$ has at least two zero eigenvalues. Those two
zero eigenvalues ensure that $A_r$ is singular by the interlacing of
eigenvalues of $A_t$ and $B_t$ (similar to Section \ref{sec:subcase1}).
So $s \le r$.

We claim that in fact $s = r$. Let $t < r$; and consider the following
system for $\Null(A_t)$:
\[
  \begin{pmatrix}
    (1-t) I + t \tilde Q & 0 & t \, \tilde g \\
    0^T & 0 & (1-t)(-\tfrac12) + t \, g_n \\
    t \, \tilde g^T & (1 - t)(-\tfrac12) + t \, g_n & t  f
  \end{pmatrix}
  \begin{pmatrix}
    \tilde z \\ z_n \\ z_{n+1}
  \end{pmatrix}
  =
  \begin{pmatrix}
    0 \\ 0 \\ 0
  \end{pmatrix}.
\]
Note that $2 g_n \le -\lambda$ and $0\leq t < r$ imply
\begin{equation} \label{equ:local1}
2 \left[ (1-t)(-\tfrac12) + t \, g_n \right]
= t(1 + 2 \, g_n) - 1 \\
\le t(1 - \lambda) - 1 \\
< r(1 - \lambda) - 1 \\
= 0,
\end{equation}
which implies $z_{n+1} = 0$. This in turn implies $\tilde z = 0$
because $(1 - t)I + t \tilde Q \succ 0$ when $t < r$. Finally, $z_n =
0$ again due to (\ref{equ:local1}). So we conclude that $t < r$ implies
$\Null(A_t) = \{0\}$. Hence, $s = r$. 
We next write
\[
  A_s = \begin{pmatrix} B_s & g_s \\ g_s & sf \end{pmatrix}.
\]
Since $\dim(\Null(B_s)) \ge 2$ and $\dim(\myspan\{g_s\}^\perp) = n-1$,
there exists $0 \ne z \in \Null(B_s)$ such that $g_s^T z = 0$. From the
structure of $B_s$, we have $z = {\tilde z \choose z_n}$, where $\tilde
z$ is a negative eigenvector of $\tilde Q$. We claim that ${z \choose 0}
\in \Null(A_s)$. Indeed:
\[
  A_s {z \choose 0}
  =
  \begin{pmatrix}
    B_s & g_s \\ g_s^T & sf
  \end{pmatrix}
  {z \choose 0}
  =
  {B_s z \choose g_s^T z} = {0 \choose 0}.
\]
Moreover, ${z \choose 0}^T A_1 {z \choose 0} = z^T B_1 z = \tilde z^T
\tilde Q \tilde z < 0$. This verifies Conditions \ref{ass:As_null} and
\ref{ass:hyperplane}.

We remark that Corollary 8 in \cite{MKV} studies the closed convex
hull of the set
\[
\left\{ y = {\tilde y \choose y_n} \in\R^n:~ \|\tilde A(\tilde{y}-\tilde{c})\|^2\leq y_n,\,
\|\tilde D(\tilde{y}-\tilde{d})\|^2\geq -\gamma \, y_n +q \right\},
\]
where $\tilde A\in\R^{(n-1)\times (n-1)}$ is an invertible matrix,
$\tilde c,\tilde d\in\R^{n-1}$ and $\gamma\geq 0$. This situation
is covered by our theory here. 
The paper \cite{BM} also applies to this case, but none of the other papers
\cite{AJ2013,BGPRT,DDV2011,K-KY14,KY14,YC14} are relevant here.

\section{Conclusion} \label{sec:conclusion}

This paper provides basic convexity results regarding the intersection
of a second-order-cone representable set and a nonconvex quadratic.
Although several results have appeared in the prior literature, we unify
and extend these by introducing a simple, computable technique for
aggregating (with nonnegative weights) the inequalities defining the two
intersected sets. The underlying conditions of our theory can be checked
easily in many cases of interest.

Beyond the examples detailed in this paper, our technique can be used in
other ways. Consider for example, a general quadratically constrained
quadratic program, whose objective has been linearized without loss of
generality. If the constraints include an ellipsoid constraint, then
our techniques can be used to generate valid SOC inequalities for the
convex hull of the feasible region by pairing each nonconvex quadratic
constraint with the ellipsoid constraint one by one. The theoretical and
practical strength of this technique is of interest for future research,
and the techniques in \cite{Androulakis.et.al.1995,Kim.Kojima.2001} could
provide a good point of comparison.

In addition, it would be interesting to investigate whether our
techniques could be extended to produce valid inequalities or explicit
convex hull descriptions for intersections involving multiple
second-order cones or multiple nonconvex quadratics. After our initial
June 2014 submission of this paper, a similar aggregation idea has been
recently explored in \cite{MV14} in November 2014 by using the results
from \cite{Yildiran09}. We note that as opposed to our emphasis on
the computability of SOCr relaxations, these recent results rely on
numerical algorithms to compute such relaxations and further topological conditions
for verifying their sufficiency.

\section*{Acknowledgments}
The authors wish to thank the Associate Editor and anonymous referees for their constructive feedback which  improved the presentation of the material in this paper.  
The research of the second author is supported in part by NSF grant CMMI 1454548.


\bibliographystyle{abbrv}
\bibliography{references}

\newpage
\setcounter{page}{1}
\setcounter{section}{0}
\setcounter{figure}{0}

\medskip
\begin{center}
{\Large {\bf Online Supplement: Low-Dimensional Examples}}\\
\medskip
\begin{multicols}{2}
{\bf Samuel Burer}\\
{\footnotesize Department of Management Sciences\\
University of Iowa,\\
Iowa City, IA, 52242-1994, USA.\\
({\tt samuel-burer@uiowa.edu})}

{\bf Fatma K{\i}l{\i}n\c{c}-Karzan}\\
{\footnotesize Tepper School of Business\\
Carnegie Mellon University,\\
Pittsburgh, PA, 15213, USA.\\
 ({\tt fkilinc@andrew.cmu.edu})}
\end{multicols}
\end{center}


In this Online Supplement, we illustrate Theorem \ref{the:main} of the main article with several
low-dimensional examples and discuss which of the earlier approaches
\cite{AJ2013,BGPRT,BM,DDV2011,K-KY14,KY14,MKV,YC14} cannot replicate
these examples. Section \ref{sec:twoterm} of the main article is devoted to the
important case for which the dimension $n$ is arbitrary, $\F_0^+$ is the
second-order cone, and $\F_1$ represents a two-term linear disjunction
$c_1^T x \ge d_1 \vee c_2^T x \ge d_2$. Section \ref{sec:general} of the main article 
investigates cases in which $\F_1$ is given by a (nearly) general
quadratic inequality.

\section{A proper split of the second-order cone} \label{sec:sub:propersplitSOC}

In $\R^3$, consider the intersection of the canonical second-order
cone, defined by \sloppy $\|(y_1;y_2)\| \le y_3$, and a specific linear
disjunction, defined by $y_1 \le -1 \vee y_1 \ge 1$, which is a
proper split. By homogenizing via $x = {y \choose x_4}$ with $x_4
= 1$ and noting that the disjunction is equivalent to $y_1^2 \ge 1
\Leftrightarrow y_1^2 \ge x_4^2$, we can represent the intersection as
$\F_0^+ \cap \F_1 \cap H^1$ with
\[
  A_0 := \Diag(1,1,-1,0), \ \ \
  A_1 := \Diag(-1,0,0,1), \ \ \
  H^1 := \{ x : x_4 = 1 \}.
\]
Note that $A_t = \Diag(1-2t,1-t,-1+t,t)$. Conditions
\ref{ass:one_neg_eval} and \ref{ass:A0_apex}(ii) are easily verified,
and Condition \ref{ass:interior} holds with $\bar x := (2;0;3;1)$,
for example.

In this case, $s = \tfrac12$, $A_s = \tfrac12 \Diag(0,1,-1,1)$,
$\F_s = \{ x : x_2^2 + x_4^2 \le x_3^2 \}$, and $\F_s^+ = \{
x : \|(x_2;x_4)\| \le x_3 \}$, which contains $\bar x$. Note
that $\apex(\F_s^+) = \Null(A_s) = \myspan\{ d \}$, where $d :=
(1;0;0;0)$. It is easy to check that $d \in H^0$ with $d^T A_1 d < 0$,
and so Conditions \ref{ass:As_null} and \ref{ass:hyperplane} are
simultaneously verified.

So, in the original variable $y$, the explicit convex hull is given by 
\[
  \left\{ y : 
  \begin{array}{l}
    \|(y_1;y_2)\| \le y_3 \\
    \|(y_2;1)\| \le y_3 
  \end{array}
  \right\}
  =
  \ccvh\left\{ y :
  \begin{array}{l}
    \|(y_1;y_2)\| \le y_3 \\
    y_1 \le -1 \vee y_1 \ge 1
  \end{array}
  \right\}.
\]
Figure \ref{fig:Ex42} depicts the original intersection, $\F_s^+\cap H^1$, and
the closed convex hull.

\begin{figure}[htp]
\centering
  \subfigure[$\F_0^+ \cap \F_1 \cap H^1$]{%
    \label{Ex42F0F1}%
    \includegraphics[width=0.30\textwidth]{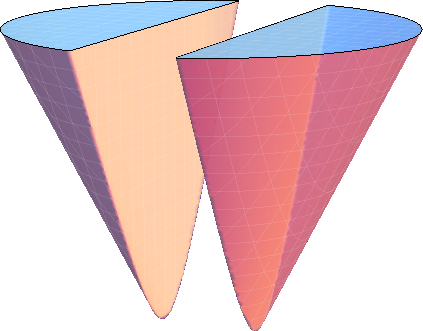}%
  } 
  \subfigure[$\F_s^+ \cap H^1$]{%
    \label{Ex42Fs}%
    \includegraphics[width=0.36\textwidth]{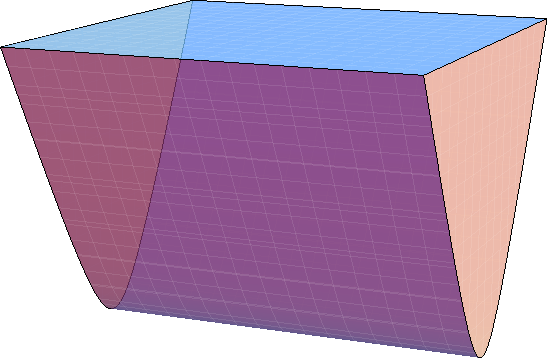}%
  }
  \subfigure[$\F_0^+ \cap \F_s^+ \cap H^1$]{%
    \label{Ex42F0Fs}%
    \includegraphics[width=0.30\textwidth]{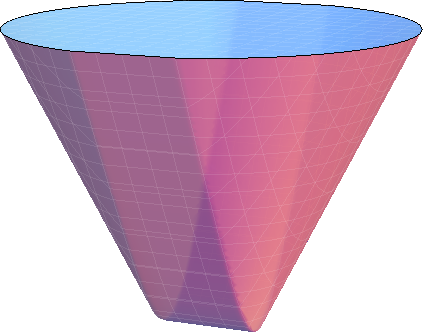}%
  }
\caption{A proper split of the second-order cone}
\label{fig:Ex42}
\end{figure}

Papers \cite{AJ2013,K-KY14,KY14,MKV} can handle this example, and in
fact they can handle all split disjunctions on SOCs. On the other
hand, \cite{BGPRT} cannot handle this example because of their
boundedness assumption on the sides of the disjunction. Because this
example concerns a disjunction on SOC itself---not a disjunction on a
cross-section of SOC---the papers \cite{DDV2011,YC14} are not relevant
here. In order to apply the results from \cite{BM}, we need to consider
the SOC $\|(y_1;y_2)\| \le y_3$ as the epigraph of the convex norm
$\|(y_1;y_2)\|$. However, this viewpoint does not satisfy the special
conditions for polynomial-time separability, such as differentiability
or growth rate, in that paper; see Theorem IV therein.

\section{A paraboloid and a second-order-cone disjunction} \label{sec:sub:parabSOC}

In $\R^3$, consider the intersection of the paraboloid defined by
$y_1^2 + y_2^2 \le y_3$ and the ``two-sided'' second-order cone
disjunction defined by $y_1^2 + y_3^2 \le y_2^2$. One side has $y_2 \ge
0$, while the other has $y_2 \le 0$. By homogenizing via $x = {y \choose
x_4}$ with $x_4 = 1$, we can represent the intersection as $\F_0^+ \cap
\F_1 \cap H^1$ with
\[
  A_0 :=
  \begin{pmatrix}
    1         & 0 & 0 & 0         \\
    0         & 1 & 0 & 0         \\
    0         & 0 & 0 & -\tfrac12 \\
    0 & 0 & -\tfrac12 & 0
  \end{pmatrix}, \ \ \
  A_1 :=
  \begin{pmatrix}
    1 &  0 & 0 & 0 \\
    0 & -1 & 0 & 0 \\
    0 &  0 & 1 & 0 \\
    0 &  0 & 0 & 0 
  \end{pmatrix}, \ \ \
  H^1 := \{ x : x_4 = 1 \}.
\]
Conditions \ref{ass:one_neg_eval} and \ref{ass:A0_apex}(i) are
straightforward to verify, and Condition \ref{ass:interior} is
satisfied with $\bar x = (0;\tfrac{1}{\sqrt2};\tfrac{1}{\sqrt3};1)$, for
example. We can also calculate $s = \tfrac12$ from (\ref{equ:def:s}). Then
\[
  A_s =
  \begin{pmatrix}
    1 &  0 & 0 & 0 \\
    0 &  0 & 0 & 0 \\
    0 &  0 & \tfrac12 & -\tfrac14 \\
    0 &  0 & -\tfrac14 & 0 
  \end{pmatrix}, \ \ \ \ \
  \F_s = \left\{ x : x_1^2 + \tfrac12 \, x_3^2 \le \tfrac12 \, x_3 x_4 \right\}.
\]
The negative eigenvalue of $A_s$ is $\lambda_{s1} := (1-\sqrt{2})/4$
with corresponding eigenvector $q_{s1} := (0;0;\sqrt{2}-1;1)$, and so,
in accordance with the Section \ref{sec:socr}, we have that $\F_s^+$
equals all $x \in \F_s$ satisfying $b_s^T x \ge 0$, where
\[
  b_s := (-\lambda_{s1})^{1/2} q_{s1}
  = \frac{\sqrt{\sqrt{2}-1}}{2}
  \begin{pmatrix}
    0 \\ 0 \\ \sqrt{2} - 1 \\ 1
  \end{pmatrix}.
\]
Scaling $b_s$ by a positive constant, we thus have
\[
  \F_s^+ := \left\{ x : \begin{array}{ll}
      x_1^2 + \tfrac12 \, x_3^2 \le \tfrac12 \, x_3 x_4 \\
      (\sqrt{2}-1) x_3 + x_4 \ge 0
    \end{array}
  \right\}.
\]
Note that $\bar x \in \F_s^+$. In addition, $\apex(\F_s^+) = \Null(A_s)
= \myspan\{d\}$, where $d = (0;1;0;0)$. Clearly, $d \in H^0$ and
$d^T A_1 d < 0$, which verifies Conditions \ref{ass:As_null} and
\ref{ass:hyperplane} simultaneously. Setting $x_4 = 1$ and returning to
the original variable $y$, we see
\[
  \left\{ y : 
  \begin{array}{l}
    y_1^2 + y_2^2 \le y_3 \\
    y_1^2 + \tfrac12 \, y_3^2 \le \tfrac12 \, y_3
  \end{array}
  \right\}
  =
  \ccvh\left\{ y :
  \begin{array}{l}
    y_1^2 + y_2^2 \le y_3 \\
    y_1^2 + y_3^2 \le y_2^2
  \end{array}
  \right\},
\]
where the now redundant constraint $(\sqrt{2}-1) y_3 + 1 \ge 0$ has been
dropped. Figure \ref{fig:Ex43} depicts the original intersection, $\F_s^+\cap H^1$, and
the closed convex hull.

\begin{figure}[htp]
\centering
  \subfigure[$\F_0^+ \cap \F_1 \cap H^1$]{%
    \label{Ex43F0F1}%
    \includegraphics[width=0.305\textwidth]{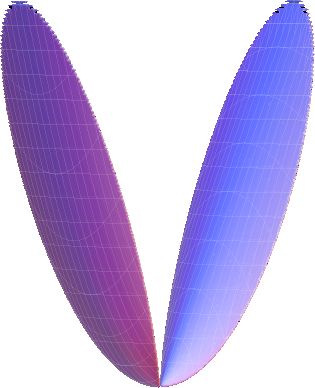}%
  } 
  \subfigure[$\F_s^+ \cap H^1$]{%
    \label{Ex43Fs}%
    \includegraphics[width=0.35\textwidth]{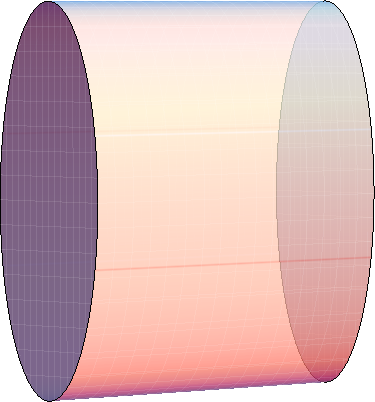}%
  }
  \subfigure[$\F_0^+ \cap \F_s^+ \cap H^1$]{%
    \label{Ex43F0Fs}%
    \includegraphics[width=0.305\textwidth]{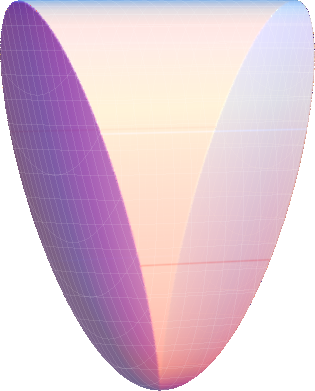}%
  }
\caption{A paraboloid and a second-order-cone disjunction}
\label{fig:Ex43}
\end{figure}

Of the earlier, related approaches, this example
can be handled by \cite{MKV} only. In particular,
\cite{AJ2013,BGPRT,DDV2011,K-KY14,KY14,YC14} cannot handle this example
because they deal with only split or two-term disjunctions but cannot
cover general nonconvex quadratics. The approach of \cite{BM} is based
on eliminating a convex region from a convex epigraphical set, but this
example removes a nonconvex region (specifically, $\R^n\setminus\F_1$).
So \cite{BM} cannot handle this example either.

In actuality, the results of \cite{MKV} do not handle this example
explicitly since the authors only state results for: the removal of a
paraboloid or an ellipsoid from a paraboloid; or the removal of an ellipsoid
(or an ellipsoidal cylinder) from another ellipsoid with a common
center. However, in this particular example, the function obtained from
the aggregation technique described in \cite{MKV} is convex on all of
$\R^3$. Therefore, their global convexity requirement on the aggregated
function is satisfied for this example.

\section{An example violating Condition \ref{ass:A0_apex}}\label{sec:sub:ex:violA0}

In $\R^2$, consider the intersection of the canonical second-order
cone defined by $|y_1| \le y_2$ and the set defined by the quadratic
$y_1(y_2-1) \le 0$. By homogenizing via $x = {y \choose x_3}$ with $x_3
= 1$, we can represent the set as $\F_0^+ \cap \F_1 \cap H^1$ with
\[
  A_0 :=
  \begin{pmatrix}
    1 &  0 & 0 \\
    0 & -1 & 0 \\
    0 &  0 & 0
  \end{pmatrix}, \ \ \
  A_1 :=
  \begin{pmatrix}
     0 &  1 & -1 \\
     1 &  0 &  0 \\
    -1 &  0 &  0 
  \end{pmatrix}, \ \ \
  H^1 := \{ x : x_3 = 1 \}.
\]
While Conditions \ref{ass:one_neg_eval} and \ref{ass:interior} hold,
Condition \ref{ass:A0_apex} does not hold because $A_0$ is singular and
$A_1$ is zero on the null space $\myspan\{(0;0;1)\}$ of $A_0$. Figure
\ref{fig:viol_Assum_3} depicts $\F_0^+ \cap \F_1 \cap H^1$ and
$\F_0^+ \cap \F_1$.

\begin{figure}[htp]
\centering
  \subfigure[$\F_0^+ \cap \F_1 \cap H^1$]{%
    \label{viol_Assum_3_F0F1H1}%
    \includegraphics[width=0.4\textwidth]{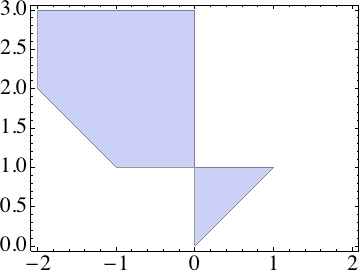}%
  } 
  \hspace*{0.25in}
  \subfigure[$\F_0^+ \cap \F_1$]{%
    \label{viol_Assum_3_F0F1}%
    \includegraphics[width=0.4\textwidth]{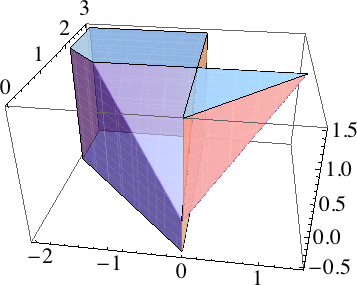}%
  }
\caption{An example violating Condition \ref{ass:A0_apex}}
\label{fig:viol_Assum_3}
\end{figure}

In this example, even though Condition \ref{ass:A0_apex} is violated,
we still have the trivial convex relaxation given by $\ccvh(\F_0^+ \cap
\F_1 \cap H^1)\subseteq \F_0^+ \cap H^1$. Of course, this trivial convex
relaxation is not sufficient.

The papers \cite{AJ2013,BGPRT,DDV2011,K-KY14,KY14,YC14} also cannot
handle this example because they deal with only split or two-term
disjunctions that are not general enough to cover general nonconvex
quadratics. Moreover, $\R^2\setminus\F_1$ defines a nonconvex region, so
neither of the approaches from \cite{BM,MKV} related to excluding convex
sets is applicable in this case.

\section{An example violating Condition \ref{ass:As_null}}
\label{sec:sub:ex:violA4}

In $\R^2$, consider the intersection of the second-order cone defined
by $|x_1| \le x_2$ and the two-term linear disjunction defined
by $x_1 \le 0 \, \vee \, x_2 \le x_1$. Note that, in the second-order
cone, $x_2 \le x_1$ implies $x_1 = x_2$. So one side
of the disjunction is contained in the boundary of the second-order
cone. We also note that---in the second-order cone---the disjunction is
equivalent to the quadratic $x_1(x_2 - x_1) \le 0$. Thus, to compute the
closed conic hull of the intersection of cone and the disjunction, we
define
\[
  A_0 := \begin{pmatrix} 1 & 0 \\ 0 & -1 \end{pmatrix},
  \ \ \ 
  A_1 :=
  \begin{pmatrix} -2 & 1 \\ 1 & 0 \end{pmatrix},
\]
and we wish to calculate $\ccnh(\F_0^+ \cap \F_1)$.

Conditions \ref{ass:one_neg_eval}, \ref{ass:interior}, and
\ref{ass:A0_apex}(i) are easily verified, and the eigenvalues of
$A_0^{-1} A_1$ are $-1$ (with multiplicity 2). This implies $s =
\tfrac12$ by (\ref{equ:def:s}), and so
\[
  A_s = \frac12 \begin{pmatrix} -1 & 1 \\ 1 & -1 \end{pmatrix}.
\]
Also, $\Null(A_s)$ is spanned by $d = (1;1)$, and yet $d^T A_1 d = 0$,
which violates Condition \ref{ass:As_null}.

\begin{figure}[htp]
\centering
  \subfigure[$\F_0^+ \cap \F_1$]{%
    \label{viol_Assum_4_F0F1}%
    \includegraphics[width=0.28\textwidth]{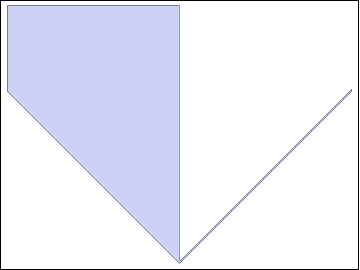}%
  } 
  \hspace*{0.2in}
  \subfigure[$\F_s^+$]{%
    \label{viol_Assum_4_Fs}%
    \includegraphics[width=0.28\textwidth]{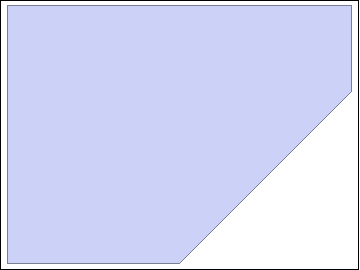}%
  }
  \hspace*{0.2in}
  \subfigure[$\F_0^+ \cap \F_s^+$]{%
    \label{viol_Assum_4_F0Fs}%
    \includegraphics[width=0.28\textwidth]{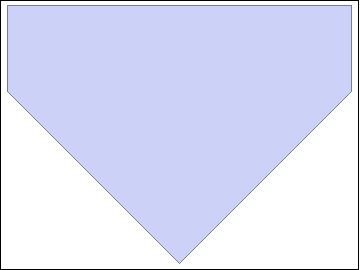}%
  }
\caption{An example violating Condition \ref{ass:As_null}}
\label{fig:viol_Assum_4}
\end{figure}

Note that $A_s = -\tfrac12 {1 \choose -1}{1 \choose -1}^T$, and so
$\F_s^+ = \{ x : x_2 \ge x_1 \}$. Figure \ref{fig:viol_Assum_4}
depicts $\F_0^+ \cap \F_1$, $\F_s^+$, and $\F_0^+ \cap \F_s^+$. Since
Conditions \ref{ass:one_neg_eval}--\ref{ass:A0_apex} are satisfied, we
know that $\ccnh(\F_0^+ \cap \F_1) \subseteq \F_0^+ \cap \F_s^+$, and it
is evident from the figures that---in this particular example---equality
holds. This simply indicates that the results of Theorem \ref{the:main}
may still hold even when Condition \ref{ass:As_null} is violated.

The approach \cite{AJ2013} can only handle split disjunctions on SOCs
and thus is not applicable here. This is also the case for that portion
of the approach from \cite{MKV} associated with split disjunctions.
Moreover, \cite{K-KY14,KY14} cannot handle this two-term disjunction
because of their strict feasibility assumption on both sides of the
sets defined by the disjunction. Also, \cite{BGPRT} cannot handle
this example because of their boundedness assumption on both of the
sets defined by the disjunction. In addition, $\R^2\setminus\F_1$
defines a nonconvex region, therefore neither of the approaches from
\cite{BM,MKV} related to excluding convex sets is applicable in this
case. Since this example concerns a disjunction on SOC itself but not on
the cross-section of an SOC, \cite{DDV2011,YC14} are not relevant here.

\section{An example violating Condition \ref{ass:hyperplane}}
\label{sec:sub:ex:violA5}

In $\R^2$, consider the intersection of the second-order cone defined
by $|y_1| \le y_2$ and the two-term linear disjunction defined by
$y_1\geq 2 \,\vee\, y_2 \leq 1$. Note that, in the second-order cone,
the disjunction is equivalent to the quadratic $(y_1 - 2)(1 - y_2) \le
0$. Thus, to compute the closed conic hull of the intersection of cone
and the disjunction, we define $x = {y \choose x_3}$ and
\[
  A_0 := \begin{pmatrix} 1 & 0 & 0 \\ 0 & -1 & 0 \\ 0 & 0 & 0 \end{pmatrix},
  \ \ \ 
  A_1 :=
  \frac12 \begin{pmatrix} 0 & -1 & 1 \\
   -1 & 0 & 2 \\
   1 & 2 & -4 \end{pmatrix}, \ \ \
   H^1 := \{ x : x_3 = 1 \}
\]
and we wish to calculate $\ccnh(\F_0^+ \cap \F_1 \cap H^1)$.

Conditions \ref{ass:one_neg_eval}, \ref{ass:interior}, and
\ref{ass:A0_apex}(iii) are easily verified, and so $s=0$ with
$\Null(A_s)$ spanned by $d = (0;0;1)$. Then Condition
\ref{ass:As_null} is clearly satisfied. However, $d_3 \ne 0$, and so the
first option for Condition \ref{ass:hyperplane} is not satisfied. The
second option is the containment $\F_0^+ \cap \F_s^+ \cap H^0 \subseteq
\F_1$, which simplifies to $\F_0^+ \cap H^0 \subseteq \F_1$ in this case. This
is also {\em not\/} true because the point $x = (1;2;0) \in \F_0^+ \cap
H^0$ but $x \not\in \F_1$.

\begin{figure}[htp]
\centering
  \subfigure[$\F_0^+ \cap \F_1 \cap H^1$]{%
    \label{viol_Assum_5_F0F1H1}%
    \includegraphics[width=0.33\textwidth]{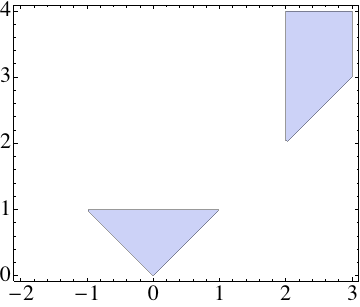}%
  } 
  \hspace*{0.2in}
  \subfigure[$\F_0^+ \cap \F_1$]{%
    \label{viol_Assum_5_F0F1}%
    \includegraphics[width=0.33\textwidth]{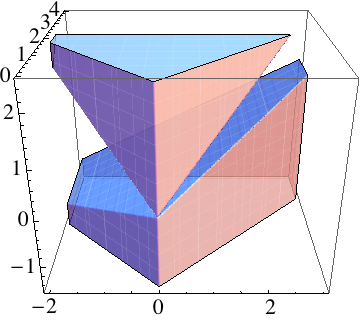}%
  }  \\
  \subfigure[$\F_s^+ = \F_0^+ \cap \F_s^+$]{%
    \label{viol_Assum_5_Fs}%
    \includegraphics[width=0.33\textwidth]{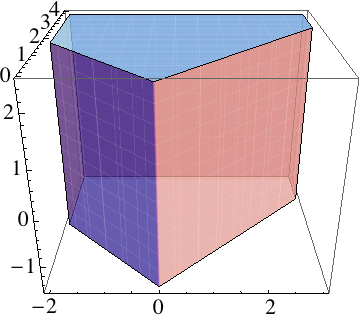}%
  } 
  \hspace*{0.2in}
  \subfigure[$\F_0^+ \cap \F_s^+ \cap H^1$]{%
    \label{viol_Assum_5_F0FsH1}%
    \includegraphics[width=0.33\textwidth]{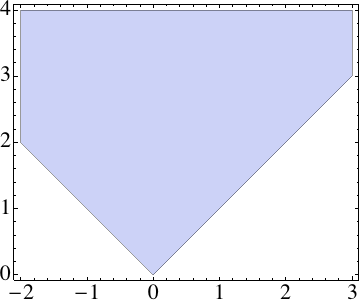}%
  }
\caption{An example violating Condition \ref{ass:hyperplane}. Note
that $s=0$ in this case.}
\label{fig:viol_Assum_5}
\end{figure}

Figure \ref{fig:viol_Assum_5} depicts this example.
Note that the inequality $y_1 \ge -1$ is valid for the convex
hull of $\F_0^+ \cap \F_1 \cap H^1$. In addition, $\F_0^+ \cap
\F_s^+ =\ccnh(\F_0^+ \cap \F_1)$ because Conditions
\ref{ass:one_neg_eval}--\ref{ass:As_null} are satisfied. However, 
the projection $\F_0^+ \cap \F_s^+ \cap H^1$ is not the desired
convex hull since, for example, it violates $y_1 \ge -1$.

Similar to the previous example in Section \ref{sec:sub:ex:violA4}, the
papers \cite{AJ2013,BGPRT,BM,DDV2011,MKV,YC14} cannot handle this
example. On the other hand, \cite{K-KY14,KY14} provide the infinite
family of convex inequalities describing the closed convex hull of this
set, but they do not specifically identify the corresponding finite
collection that is necessary and sufficient.

\end{document}